\documentclass{amsart}
\usepackage[foot]{amsaddr}
\usepackage{bbm,amsmath,amsthm,amsfonts, amssymb,mathdots,enumitem,mathrsfs,color}

\usepackage{hyperref}

 \newtheorem{thm}{Theorem} [section]
 \newtheorem*{theorem*}{Theorem}
\newtheorem{definition}[thm]{Definition}
\newtheorem{lem}[thm]{Lemma}
 \newtheorem{prop}[thm]{Proposition}
\newtheorem{cor}[thm]{Corollary}

 \newtheorem{conj}[thm]{Conjecture}
 \newtheorem{prob}[thm]{Problem}

\theoremstyle{remark}

\newcommand{\mP}{\mathcal{P}}  \newcommand{\mQ}{\mathcal{Q}}
  
\newcommand{\mO}{\mathcal{O}}  \newcommand{\mA}{\mathcal{A}}
\newcommand{\mAx}{\mathcal{AX}}
\newcommand{\mC}{\mathcal{C}}
\newcommand{\mT}{\mathcal{T}} 
\newcommand{\mRC}{\mathcal{RC}}
\newcommand{\mIC}{\mathcal{IC}}
\newcommand{\mE}{\mathcal{E}}
\newcommand{\mF}{\mathcal{F}}
\newcommand{\mM}{\mathcal{M}}  \newcommand{\mN}{\mathcal{N}}

\newcommand{\bZ}{\mathbb{Z}}
 
 \newcommand{\bP}{\mathbb{P}} 
 \newcommand{\bF}{\mathbb{F}}  
 \newcommand{\fO}{\mathfrak{O}}  
 \newcommand{\ep}{\epsilon} 

\newcommand{\bbbm}{ \begin{bmatrix}} \newcommand{\bebm}{\end{bmatrix}}
\newcommand{\bbm}{ \begin{pmatrix}} \newcommand{\bem}{\end{pmatrix}}
\newcommand{\bbsm}{ \left( \begin{smallmatrix}} \newcommand{\besm}{\end{smallmatrix} \right)}
\newcommand{\beq}{\begin{equation}}      \newcommand{\eeq}{\end{equation}}
\newcommand{\beqn}{\begin{eqnarray}}      \newcommand{\eeqn}{\end{eqnarray}}
\newcommand{\beqs}{\begin{eqnarray*}}      \newcommand{\eeqs}{\end{eqnarray*}}
\newcommand{\bep}{\begin{proof}}      \newcommand{\eep}{\end{proof}}

\begin{document}
\title{On the $PGL_2(q)$-orbits of lines of $PG(3,q)$  and binary quartic forms  in characteristic three}
\author{Krishna Kaipa$^1$ }
\address{$^1$Department of Mathematics, Indian Institute of Science Education and Research, Pune, Maharashtra, 411008 India.}
\address{$^2$Department of Mathematics, Indian Institute of Technology Bombay, Mumbai 400076, India
}
\address{$^{\ast}$Corresponding author}

\author{Puspendu Pradhan$^{2,\ast}$}
\email{$^1$kaipa@iiserpune.ac.in, $^2$puspendupradhan1@gmail.com}

\date{}

\begin{abstract}
We consider the problem of classifying the lines of  the projective $3$-space $PG(3,q)$ over a finite field $\bF_q$  into orbits of the group $PGL_2(q)$ of linear symmetries of the twisted cubic $C$. The problem has been solved in literature  in characteristic different from $3$, and in this work,  we solve the problem  in characteristic $3$.  We reduce this problem to another problem, which is the classification of binary quartic forms into $PGL_2(q)$-orbits. We first solve the latter problem and use to solve the former problem. We also obtain the point-line and the line-plane incidence structures of the point, line, and plane orbits.
\end{abstract}
 \maketitle
\section{Introduction} \label{introduction}
Let $V$ be a vector space of dimension $(n+1)$ over a field $F$. Let $\bP(V)$ denote the $n$-dimensional projective space associated with it. Let $PGL_{n+1}(V)$ denote the group of projectivities of $\bP(V)$. When $F=\bF_q$, a finite field of order $q$, we denote the associated $n$-dimensional projective space by $PG(n,q)$ and the group of projectivities by $PGL_{n+1}(q)$. Let $C$ be a \emph{twisted cubic curve} in  $PG(3,q)$  defined by the image of the embedding 
\[PG(1,q) \hookrightarrow PG(3,q), \quad  (s,t) \mapsto (s^3,s^2t,st^2,t^3). \]
For $q\geq 5$, the subgroup of $PGL_4(q)$ which fixes $C$ is isomorphic to $PGL_2(q)$ which we denote as $G$. The group $G$ acts triply transitively on the points of $C$, since its action on $PG(1,q)$ is triply transitive. The \emph{osculating plane} at a point $(s^3,s^2t,st^2,t^3)$ of $C$ is given by  \[t^3X_0-3st^2X_1+3s^2tX_2-s^3X_3=0.\]

 We say that a line $L$  of $PG(3,q)$ is \emph{generic} if it neither intersects $C$, nor lies in any of the osculating planes of $C$. We say a line is \emph{non-generic} if it is not generic. 
 
 The problem studied in this paper is:
\begin{prob}  \label{mainproblem}
Classify the generic lines of $PG(3,q)$ into $G$-orbits.
\end{prob} 
This problem was recently solved  in \cite{KPP} in the case when characteristic of the field $\bF_q$ is different from $2$ and $3$. In the case of characteristic $2$, the problem was solved by M. Ceria and F. Pavese in  \cite{Ceria_Pavese}. In this work, we solve the problem for the remaining case of characteristic $3$. Throughout this work, we assume char$(\bF_q)=3$. 
 The class of  generic lines is denoted as $\mO_6$ in the book   \cite[p.236]{Hirschfeld3}. 
The classification of non-generic lines  into ten $G$-orbits is well-known in literature,   for example \cite[Theorem 8.1]{BPS},   \cite[p.236]{Hirschfeld3},   \cite[Theorem 3.1]{DMP1}. 
It has been an open problem to classify  the class $\mO_6$ of generic lines into $G$-orbits.
A detailed conjecture was made by A. Davydov, S. Marcugini and F. Pambianco in  \cite[Conjecture 8.2]{DMP1}. The part of this conjecture relevant to characteristic $3$ is:
\begin{conj} \label{conj1}\cite{DMP1} The number of $G$-orbits of  generic lines of $PG(3,q)$ is $(2q-3)$,    of which  there are $q/3, (q-1)$ and $(2q-6)/3$  orbits of size $|G|, |G|/2$ and $|G|/4$, respectively.
\end{conj}

Our main result is: 
\begin{theorem*} \label{main} The above conjecture is true.
\end{theorem*}
The detailed result is given in Theorem \ref{thm_main} below. We briefly explain our approach to the problem: A binary quartic form 
\[f(X,Y)=  z_0 X^4+z_1 X^3Y +z_2 X^2Y^2 +z_3 XY^3+ z_4Y^4,\]
is a homogeneous polynomial of degree $4$ in two variables $X, Y$. The group $GL_2(q)$ acts on such forms by 
\[ g  f(X,Y)=\det(g)^{-4} f(dX-bY, aY-cX), \quad g = \bbsm a&b\\c&d\besm. \]
This gives an action of the group $G=PGL_2(q)$ on the projective space of binary quartic forms.
The form $f(X,Y)$ has repeated factors if and only if its discriminant
\[ \Delta(f) = (z_0z_4-z_1z_3)^3+(z_0z_4-z_1z_3)^2 z_2^2 - z_2^3(z_0z_3^2+z_4z_1^2-z_0z_2z_4),\]
vanishes. The $\jmath$-invariant of a form with $\Delta(f) \neq 0$ is given by 
\[\jmath(f) = z_2(f)^6/\Delta(f).\]
We will show that:\\

\emph{There is a $G$-equivariant correspondence between the set of generic lines of $PG(3,q)$ and the set of  binary quartic forms over $\bF_q$ whose discriminant and $\jmath$-invariant are nonzero.}\\

This correspondence arises as follows: A generic line $L$ intersects  the tangent line to $C$ at four distinct points $(s^3,s^2t,st^2,t^3)$ of $C(\overline{\bF_q})$ if and only if the associated quartic form $f_L$ vanishes at $(s,t)$, that is, $(Xt-Ys)$ is a factor of $f_L(X,Y)$.  Therefore,  Problem \ref{mainproblem}  is equivalent to 
decomposing the projective space of binary quartic forms with non-zero discriminant and $\jmath$-invariant  into $G$-orbits.\\

There are two  problems closely related to Problem \ref{mainproblem} as follows: First, we recall that the points of $PG(3,q)$ decompose into five $G$-orbits denoted $\mC, \mAx, \mT, \mRC, \mIC$, and similarly the planes of $PG(3,q)$ decompose as $\amalg_{i=1}^5 \mathcal{N}_i$ into five $G$-orbits,  (see Propositions  \ref{point_orbits} and \ref{plane_orbits}, respectively).
\begin{prob} \label{prob1} For representatives $L$ of each of the line orbits, determine the proportion of the set $\mP$ of the $(q+1)$ points of $L$, lying in each of the $5$ point orbits:
\[ \frac{|\mathcal \mP \cap \mM|}{q+1}, \quad \mM \in \{\mC, \mAx, \mT, \mRC, \mIC \}.\]
\end{prob}

\begin{prob} \label{prob2} For representatives $L$ of each of the line orbits, determine the proportion of the set $S$ of the pencil of $(q+1)$ planes containing $L$, lying in each of the $5$ plane  orbits:
\[ \frac{|S \cap \mathcal N_i|}{q+1}, \quad i=1,\dots ,5.\]
\end{prob}

When $\text{char}(\bF_q)\neq 3$, Problem \ref{prob1} and \ref{prob2} are equivalent, because there is a $G$-equivariant  polarity induced by $C$.  When $\text{char}(\bF_q)\neq 2,3$ both these problems have been solved for the orbits of non-generic lines by Davydov, Marcugini and Pambianco in \cite{DMP3,DMP5}, G\"unay and Lavrauw in \cite{GL}, while  the problem for generic lines remains open. When $\text{char}(\bF_q)= 2$, both the problem have been solved for all lines (generic as well as non-generic) by Ceria and Pavese in \cite{Ceria_Pavese}. When $\text{char}(\bF_q)= 3$, Problem \ref{prob1} and \ref{prob2}  for all the orbits of non-generic lines have been solved independently by A. Davydov, S. Marcugini and F. Pambianco in \cite{DMP3, DMP5}. 
In this work, we solve both the problems for all the orbits of generic lines when $\text{char}(\bF_q)= 3$. The detailed result is given in Theorem \ref{result3} and \ref{result4}.

\subsection{Statement of results}~\\ \label{statement}
We need some more notation in order to state the results: We partition $\bF_q^\times = J_4 \cup J_2 \cup J_1$ where 
\begin{align*}
    J_2&=\{ r\in \bF_q^\times \colon r  \text{ is a non-square in $\bF_q$}\},\\
    J_4&=\{ r\in \bF_q^\times \colon r  \text{ is a square in $\bF_q$ and Tr$_{\bF_q/\bF_3}(1/\sqrt{r})=0$}\},\\
J_1&=\{ r\in \bF_q^\times \colon r  \text{ is a square in $\bF_q$ and Tr$_{\bF_q/\bF_3}(1/\sqrt{r})\neq 0$}\}.
\end{align*}
The sizes of the sets $J_i$ are (see Lemma \ref{Ji_lem}): 
\[ |J_4|=(q-3)/6, \quad |J_2|=(q-1)/2, \quad |J_1|=q/3. \]

  The set  of binary quartic forms with non-zero discriminant can be partitioned as $\mF_1 \cup \mF_2 \cup \mF_2' \cup \mF_4 \cup \mF_4'$ based on the factorization of $f$ over $\bF_q[X,Y]$. The irreducible factors (over $\bF_q$) in each case are: 
  \begin{description}
      \item[$\mF_4:$]  $4$ linear forms,
      \item[$\mF_2:$] $2$ linear forms and a quadratic form,
\item[$\mF_1:$] $1$ linear forms and a cubic form,
 \item[$\mF_4':$]  $2$ quadratic forms,
 \item[$\mF_2':$] $1$ quartic form (that is, $f$ is irreducible over $\bF_q$).    
  \end{description}
  
\begin{thm} \label{result2} The set of $(q^4-q^2)$  binary quartic forms over $\bF_q$ with non-zero discriminant, decomposes into the following $G$-orbits parametrized by the $\jmath$-invariant.  The number of orbits of size indicated by the column index, and of $\jmath$-invariant indicated by the row index, is given by the corresponding entry of the table below. The total number of orbits is $2q+2$.
\begin{table}[h!]  
\begin{tabular}{c| *{6}{c}}
   & $|G|$ & $\tfrac{|G|}{2}$ & $\tfrac{|G|}{3}$ & $\tfrac{|G|}{4}$ & $\tfrac{|G|}{8}$    &  $\tfrac{|G|}{24}$  \\
&&&&&&\\  \hline \\
$\jmath(f)\in J_4$ & $0$ & $0$ & $0$ & $4$ & $0$ &  $0$  \\
$\jmath(f)\in J_2$ & $0$ & $2$ & $0$ & $0$ & $0$ &  $0$  \\
$\jmath(f)\in J_1$ & $1$ & $0$ & $0$ & $0$ & $0$ &  $0$  \\
$\jmath(f)=0$  & $0$ & $0$ & $1$ & $2$ &  $1$  &  $1$
  \\ [1ex]
\hline  \\ [1ex]
\end{tabular}
\caption{Table of $G$-orbits of quartic forms with nonzero discriminant}
\label{table:1}
\end{table}
For $i \in \{1,2,4\}$,  for each $r \in J_i$,  there are $i$ orbits of size $|G|/i$ that have $\jmath$-invariant $r$.  For   $i \in \{2,4\}$,  of the $i$ orbits, there is one orbit in $\mF_i$ and $(i-1)$ orbits in $\mF_i'$. There is one orbit with $\jmath$-invariant zero in each of the types $\mF_1,\mF_2,\mF_4,\mF_2',\mF_4'$.
\end{thm}
A list of representatives of these $(2q+2)$-orbits together with the isomorphism class of the stabilizer subgroup of each orbit  is given in Table \ref{table-BQF}.\\

\begin{thm} \label{thm_main}
The number of $G$-orbits $\fO$ of generic lines of  size indicated by the column index and 
$\jmath$-invariant indicated by the row index is given in the table below.
\begin{table}[h!]
\noindent \begin{tabular}{c| *{3}{c}}
$\jmath(\fO)$   & $|G|$ & $\tfrac{|G|}{2}$ & $\tfrac{|G|}{4}$    \\
  \hline \\
$\jmath\in J_4$ & $0$ & $0$  & $4$   \\
$\jmath\in J_2$ & $0$ & $2$ & $0$    \\
$\jmath\in J_1$ & $1$ & $0$ & $0$  
  \\ [1ex]
\hline
\\
\emph{$\#$ of orbits} & $\tfrac{q}{3}$ &  \small{$q-1$} & $\tfrac{2q-6}{3}$. \\ [1ex]
\end{tabular}
\caption{Table of $G$-orbits of generic lines}
\label{table:2}
\end{table}
 \end{thm}
In other words, for each $i \in \{1, 2, 4\}$, there are $i |J_i|$ orbits of size $|G|/i$. In Table \ref{table-Lines}, we list the generators in $PG(3,q)$ of each of these line orbits.\\

In order to state our results
about Problems \ref{prob1} and \ref{prob2}, we need one more definition. Let $\mE_r$ denote the elliptic curve in $\bP^2$ given by
\[ \mE_r: x^2 = y^3+ y^2- r^{-1},\]
and let $\# \mE_r(\bF_q)$ denote the number of points of $\mE_r$ over $\bF_q$. The $\jmath$-invariant of the elliptic curve $\mE_r$ in characteristic $3$  is also $r$ (following  the definition in \cite[Appendix A, Proposition 1.1]{Silverman_AEC}).

\begin{thm} \label{result3}
Let $L$ be a generic line and let $S$ denote the set of planes through $L$.
Let  $f_L$ denote the binary quartic form associated with $L$, and let $r=\jmath(f_L)$. \\
\begin{enumerate}
\item In case  $f_L \in \mF_i$ for $i \in \{1, 2, 4\}$,  we have 
\begin{enumerate}
        \item $|S\cap \mN_1|=0$,
        \item $|S\cap \mN_2|=i$,
        \item $|S\cap \mN_4|=q+1-\frac{\#\mE_r(\bF_q)+i}{2}$,
\item $|S\cap \mN_3|=\tfrac{\#\mE_r(\bF_q)-3i}{6}$,
   \item $|S\cap \mN_5|=\tfrac{\#\mE_r(\bF_q)}{3}$.    \\
    \end{enumerate}

\item In case  $f_L \in \mF_2' \cup \mF_4'$, we have:
\begin{enumerate}
        \item $|S\cap \mN_1|=|S\cap \mN_2|=0$
         \item $|S\cap \mN_4|=q+1-\frac{\#\mE_r(\bF_q)}{2}$,
\item $|S\cap \mN_3|=\tfrac{\mE_r(\bF_q)}{6}$,
 \item $|S\cap \mN_5|=\tfrac{\#\mE_r(\bF_q)}{3}$. 
         \end{enumerate}
    \end{enumerate}
\end{thm} 

We remark that the points $(x,y) = (\pm r^{-1/3}, r^{-1/3})$ are flex points of the curve $\mE_r$, and hence are $3$-torsion points of $\mE_r(\bF_q)$. Therefore,  $\#\mE_r(\bF_q)$ is divisible by $3$.

\begin{thm} \label{result4}
Let $L$ be a generic line of $PG(3,q)$ and let $\mP$ denote the set of points of $L$. Let $f_L$ be the binary quartic for associated with $L$, and let $r=\jmath(f_L)$. \\
\begin{enumerate}
\item In case  $f_L \in \mF_i$ for $i \in \{1, 2, 4\}$,  we have 
\begin{enumerate}
        \item $|\mP\cap \mC|=|\mP\cap \mAx|=0$,
        \item $|\mP\cap \mT|=i$,
        \item $|\mP\cap \mRC|= \frac{\#\mE_r(\bF_q)-i}{2}$,
        \item $|\mP\cap \mIC|=q+1- \frac{\#\mE_r(\bF_q)+i}{2}$.\\
    \end{enumerate}

\item In case  $f_L \in \mF_2' \cup \mF_4'$, we have:
\begin{enumerate}
        \item $|\mP\cap \mC|=|\mP\cap \mAx|=|\mP\cap \mT|=0$,
        \item $|\mP\cap \mRC|=\frac{\#\mE_r(\bF_q)}{2}$,
        \item $|\mP\cap \mIC|=q+1-\frac{\#\mE_r(\bF_q)}{2}$.
    \end{enumerate}
    \end{enumerate}

\end{thm}

The rest of the paper is organized as follows. In \S \ref{two} we describe the geometric setup of the problem, and we obtain the equivalence between Problem \ref{mainproblem} and  the problem of classifying $G$-orbits of binary quartic forms whose discriminant and $\jmath$-invariant are non-zero. In \S \ref{three}  we prove Theorem \ref{result2}. In \S \ref{four} we prove  Theorem \ref{thm_main}. In \S \ref{five}, we obtain a result (see Theorem \ref{thm_elliptic}), needed for our solution of Problems \ref{prob1} and \ref{prob2}.  In \S\ref{six} and \S\ref{seven}, we prove Theorems \ref{result3} and \ref{result4}, respectively.

\section{Geometric Setup} \label{two}
Let $V$ denote the vector space $F^2$ over a field $F$, and let  $e_1=(1,0), e_2=(0,1)$ 
denote the standard basis of $V$. Let $X, Y$ denote the dual basis for the dual vector space $V^*$. The group $GL_2(F)$ acts on $V$, by $g (e_1)=ae_1+ce_2$ and $g(e_2)=b e_1+ de_2$ for $g =\bbsm a & b\\ c&d \besm \in GL_2(F)$. The action of $GL_2(F)$ on the dual vector space is given by $g(X)=(dX-bY)/\text{det}(g)$, $g(Y)=(aY-cX)/\text{det}(g)$. Let Sym$^m(V^*)$ denote the $m$-th symmetric power of $V^*$. In terms of the basis $X^{m-i}Y^i, i=0,\dots,m$, we realize Sym$^m(V^*)$ as the  vector space of degree $m$ homogeneous polynomials (or binary forms) $f(X,Y)$  in $X,Y$. The action of $GL_2(F)$ on Sym$^m(V^*)$ is given by 
\[  g_m f(X,Y)=\det(g)^{-m} f(dX-bY, aY-cX). \]

Let $D_mV$ denote the dual vector space to Sym$^m(V^*)$ and let $B_0, \dots, B_m$ denote the  basis  of $D_mV$ dual to the basis $X^{m-i}Y^i, i=0,\dots,m$ of Sym$^m(V^*)$. For $g \in GL_2(F)$, we denote by $g^{[m]}$, the action of $GL_2(F)$  on $D_mV$ dual to the action $g_m$ on Sym$^m(V^*)$. For $(s,t) \in V$, the evaluation  $f \mapsto f(s,t)$ gives a map $\nu_m: V \to D_mV$ given by $\nu_m(s,t) = s^m B_0 +s^{m-1}t B_1+ \dots+ t^m B_m$. At the projective level we get the Veronese map
\[\nu_m: \bP(V) \hookrightarrow \bP(D_mV), \quad (s,t)\mapsto \sum_{i=0}^m s^{m-i}t^i B_i.\]
The image of $\nu_m$ is the degree $m$ rational normal curve (in short RNC) $C_m$. 
By  construction, 
\[ \label{eq:nu_equivariance}  \nu_m(g \cdot v) = g^{[m]} \nu_m(v), \]
or more explicitly $\nu_m(as+bt,cs+dt)=g^{[m]} (s^{m}, s^{m-1}t, \dots, s t^{m-1}, t^{m})$. \\

For the rest of this section, we assume char$(F)=3$. We denote the curve $C_3$ as just $C$. We record the matrices $g^{[3]}$ and $g_4$ for $g =\bbsm a & b\\ c&d \besm \in GL_2(F)$, in the case when  char$(F)=3$:

\beq \label{eq:g_3}
g^{[3]}=\bbm a^3 &0 &0&b^3\\
a^2c &a(ad-bc) &-b(ad-bc)&b^2d\\
ac^2 &-c(ad-bc)&d(ad-bc)&bd^2\\
c^3 &0 &0&d^3\bem, 
 \eeq
\beq \label{eq:g_4}
 g_4=\tfrac{1}{\text{det}(g)^4} \bbm d^4 &-d^3c&d^2c^2&-dc^3&c^4\\
-d^3b &d^3a & dc(bc+ad)&c^3b&-c^3a\\
0 & 0& (ad-bc)^2&0&0\\
-b^3d &b^3c& ab(bc+ad) &a^3d&-c a^3\\
b^4 &-b^3a&a^2b^2&-b a^3&a^4\bem. \eeq

It follows from \eqref{eq:g_3} that the line $\mA$ of $\bP(D_3V)$ spanned by $B_1$ and $B_2$ is $PGL_2(F)$-invariant. The osculating plane $O_{(s,t)}$ to $C$ at a point $\nu_3(s,t)$ consists of points (see \cite[eq. 21.4]{Hirschfeld3}) \[ O_{(s,t)}=\{(y_0,y_1,y_2,y_3) \colon y_0 t^3-3 t^2s y_1 +3 ts^2 y_2 -s^3 y_3 =0\}.\]  Since char$(F)=3$, we have 
\[O_{(s,t)}=\{(y_0,y_1,y_2,y_3) \colon y_0 t^3 -s^3 y_3 =0\}.  \]
We note that $O_{(s,t)}$ is the join of $\mA$ and $\nu_3(s,t)$,  and that  $\mA$ is the line of intersection of any two  and hence all osculating planes of $C$. The line $\mA$ is called the \emph{axis}. It also follows that a line $L$ of $\bP(D_3V)$ meets $\mA$ if and only if $L$  is contained in some osculating plane of $C$.\\

We also note from \eqref{eq:g_4} that the linear subspace  of $\text{Sym}^4(V^*)$ consisting of those quartic forms for which the coefficient of $X^2Y^2$ is zero, is a $PGL_2(F)$-invariant subspace. \\

The basis of the second exterior power $\wedge^2 D_3V$ associated with the basis $B_0, \dots, B_3$ of $D_3V$ is $B_{ij} = B_i \wedge B_j, 0 \leq i <j \leq 3$. The Pl\"ucker coordinates with respect to this basis are denoted by $p_{ij}, 0 \leq i < j \leq 3$. It will be convenient to use the basis
\[ E_0=B_{23}, E_1=B_{13}, E_2=B_{03}, E_3=B_{02}, E_4=B_{01}, E_5=B_{12},\]
 and let  \[(z_0, z_1,z_2,z_3,z_4,z_5)=(p_{23}, p_{13}, p_{03}, p_{02}, p_{01}, p_{12})\] be the coordinates with respect to the basis $E_0, \dots, E_5$. The lines of $\bP(D_3V)$ are given by the Klein quadric 
\[ \mQ: p_{01}p_{23}-p_{02}p_{13}+p_{03}p_{12}=0, \;\text{equivalently } z_0z_4 -z_1z_3 + z_2z_5=0. \]

The induced action of $GL_2(q)$ on $\wedge^2 D_3V$ is represented by the matrix
\[   \det(g)  \, \left( \begin{array}{c|c} 
  \det(g)^4  g_4 &  0    \\   \hline 
    b^2d^2 \text{ } bd(ad+bc) \text{ } abcd \text{ } ac(ad+bc) \text{ } a^2c^2 & \det(g)^2
\end{array} \right), \] 
with respect to the basis $E_0, \dots ,E_5$.  From the form of this matrix, the following two observations are clear. Let $\hat A= F \cdot E_5$ denote the one dimensional subspace of $\wedge^ 2 D_3V$ corresponding to the point $\mA=E_5$ of $\bP(\wedge^ 2 D_3V)$. (We  use the same symbol $\mA$ for the axis in $\bP(D_3V)$ and the corresponding point $E_5$ on the Klein quadric $\mQ$ in $\bP(\wedge^2 D_3V)$ under the Klein correspondence. The intended meaning of $\mA$ will be clear from the context in which it appears.)
\begin{enumerate}
    \item As observed above the point  $\mA$ of $\bP(\wedge^ 2 D_3V)$ is fixed by $PGL_2(F)$.  It represents the common line of intersection of all osculating planes of $C$. 
\item The quotient vector space 
\[(\wedge^2 D_3V)/\hat A  \simeq  (\wedge^2V)^{\otimes 5} \otimes \text{Sym}^4(V^*), \]
as $GL_2(F)$-modules. In particular, at the projective level we get a $PGL_2(F)$-equivariant isomorphism
\[ \bP\left(  (\wedge^2 D_3V)/\hat A \right) \simeq  \bP( \text{Sym}^4(V^*) ),\]
\[\sum_{i=0}^5 z_i E_i \mapsto z_0 X^4 +z_1 X^3Y+ z_2 X^2Y^2 + z_3 XY^3 +z_4 Y^4. \]
Here $\bP\left(  (\wedge^2 D_3V)/\hat A\right)$ is the set of lines of $\bP (\wedge^2 D_3V)$ passing through $\mA$.
In this way we assign a binary quartic form $f_L(X,Y)$ to each line $L \neq \mA$ of $\bP(D_3V)$.\\ 
\end{enumerate}

The tangent hyperplane to $\mQ$ at $\mA$  is $\Pi_{\mA}: z_2=0$. Let $\mQ_0=\Pi_{\mA} \cap \mQ$. In the Klein correspondence between points of $\mQ$ and  the lines of $\bP(D_3V)$, the points of $\mQ_0$ are precisely the set of lines intersecting the axis $\mA$. We note that $\mQ_0$ in $\Pi_{\mA}$ is the cone with vertex $\mA$ and whose base is the hyperbolic quadric $z_0z_4-z_1z_3=0$ in the solid $z_2=z_5=0$.
Let $\mQ_1=\mQ \setminus \mQ_0$, which represents the set of lines of $\bP(D_3V)$ not intersecting $\mA$.
We note that distinct points of $\mQ_1$ have distinct images in $\bP\left(  (\wedge^2 D_3V)/\hat A \right)$: given two distinct points $(z_0, \dots, z_4, (z_1z_3-z_0z_4)/z_2)$ and $(z_0', \dots, z_4', (z_1'z_3'-z_0'z_4')/z_2')$  of $\mQ_1$, the point $\mA$ does not lie on the line joining these these two points of $\mQ_1$.  Let $\iota: \mQ_1 \hookrightarrow \bP\left(  (\wedge^2 D_3V)/\hat A \right)$ denote this injective map. We note that the $\iota$ is also $PGL_2(F)$-equivariant, because $\iota(P)$ is the line joining $\mA$ and $P$, and is carried by $g \in PGL_2(F)$ to the line joining $\mA$ and $g \cdot P$, which is $\iota(g \cdot P)$. Composing $\iota$ with the $PGL_2(F)$-equivariant isomorphism $\bP\left(  (\wedge^2 D_3V)/\hat A \right) \simeq  \bP( \text{Sym}^4(V^*) )$ above, we conclude that the map $\mQ_1 \to \bP( \text{Sym}^4(V^*) )$ given by $L \mapsto f_L$ gives a $PGL_2(F)$-equivariant isomorphism between $\mQ_1$ and the set of quartic forms $f$ with $z_2(f) \neq 0$. \begin{prop} \label{correspondence}
    The $PGL_2(F)$-orbits of the set of lines of $\bP(D_3V)$ not meeting the axis $\mA$, are in bijective correspondence under the map  $L \mapsto f_L$, with the $PGL_2(F)$-orbits of binary quartic forms $f$ with $z_2(f) \neq 0$.

Further, the $PGL_2(F)$-orbits of the set of generic lines of $\bP(D_3V)$ are in bijective correspondence under the map  $L \mapsto f_L$, with the $PGL_2(F)$-orbits of binary quartic forms $f$ with $j(f) \neq 0$.
\end{prop}
\bep  The first assertion has just been proved above.
For the second assertion, we first note that the condition that $ \jmath(f_L) \neq 0$ is equivalent to $z_2(f_L)$ and $\Delta(f_L)$ both being nonzero. Next,  we show that a line $L \neq \mA$  of $\bP(D_3V)$ intersects $\mA$, if and only if $f_L$ has a linear factor of multiplicity at least $3$:  If $L$ meets $\mA$,  we can write $f_L= X^3(z_0X+z_1Y)+(z_3X+z_4Y)Y^3$ because $z_2(f_L)=0$. Using  $z_2=0$ in  the equation $z_0z_4 -z_1z_3 + z_2z_5=0$, we get  $z_0z_4-z_1z_3=0$. This implies that the linear forms $(z_0X+z_1Y)$ and $(z_3X+z_4Y)$ are dependent, and hence $f_L(X,Y)$ is of the form $\psi_1 \cdot \psi_2^3$ for some linear forms $\psi_1, \psi_2$ over $F$. Conversely, if $f_L$ has a linear form of multiplicity at least $3$, then  $f_L$ is of the form $(Xt-Ys)^3(\alpha X+ \beta Y)=(X^3t^3-Y^3s^3)(\alpha X+ \beta Y)$ which clearly has $z_2(f_L)=0$.

Therefore $L$ does not meet the axis $\mA$,  if and only if each linear factor of $f_L$ has multiplicity at most $2$.
We will now show that a quartic form $f_L$ with $z_2(f_L) \neq 0$ has a linear factor of  multiplicity $2$ if and only if  $L$ meets $C$. This will conclude the proof that generic lines of $\bP(D_3V)$ (that is, lines not meeting $C \cup \mA$) are precisely those lines for which $ j(f_L) \neq 0$, and hence $PGL_2(F)$-orbits of generic lines are in bijective correspondence under $L \mapsto f_L$ with $PGL_2(F)$-orbits of forms with $j(f_L) \neq 0$.

If $L$ passes through $\nu_3(s,t)$ but is not contained in the osculating plane $O_{(s,t)}$ (that is, $L$ does not intersect $\mA$), then replacing $L$ by $g \cdot L$  where $g \cdot (s,t)=(0,1)$, $g\in PGL_2(F)$, we may assume $L$ is generated by $(0,0,0,1)$ and $(1,\alpha, \beta, 0)$ for some $\alpha, \beta \in F$. Therefore, $L$ has  coordinates 
\[(z_0, z_1,z_2,z_3,z_4,z_5)=(p_{23}, p_{13}, p_{03}, p_{02}, p_{01}, p_{12})=(\beta, \alpha, 1,0,0,0).\] Thus $f_L=X^2(\beta X^2+\alpha XY+Y^2)$, has $X$ as a linear factor of multiplicity exactly $2$. We conclude that a line $L$ passes through $\nu_3(s,t)$ and does not meet $\mA$, has the property that $(Xt-Ys)$ is a factor of $f_L$  of multiplicity exactly $2$. 
Conversely, if $f_L=X^2(\beta X^2+\alpha XY+Y^2)$, then $p_{01}=p_{02}=p_{12}=0$ and $p_{03}=1,p_{23}=\beta,p_{13}=\alpha$ which is equivalent to $L$ being generated by the two points $(0,0,0,1)$ and $(1,\alpha,\beta,0)$ of $\bP(D_3V)$. Thus  
$(Xt-Ys)$ is a factor of $f_L$  of multiplicity exactly $2$, if and only if $L$ passes through $\nu_3(s,t)$ but does not meet $\mA$. We  conclude that  $\Delta(f_L)=0$ and $z_2(f_L) \neq 0$ if and only if $L$ intersects $C$ but not $\mA$.
\eep

The quartic form $f_L$ associated to a line $L \neq \mA$ also has the following geometric interpretation:
\begin{lem}
    Let $L$ be a line of $\bP(D_3V)$ other than the axis $\mA$, and $f_L(X,Y)$ be the binary quartic form associated with $L$. Let $\overline{F}$ denote an algebraic closure of $F$. The line  $L$ meets the tangent line to $C(\overline{F})$ at a point $\nu_3(s,t)=(s^3,s^2t,st^2,t^3)$   if and only if $f_L(s,t)=0$.
\end{lem}
\begin{proof}
    The Pl\"ucker coordinates of a tangent line $T$ at  $\nu_3(s,t)=(s^3,s^2t,st^2,t^3)$ of $C(\overline{F})$ is $(s^4,-s^3t,0,s^2t^2,-st^3,t^4)$. By \cite[Lemma 15.2.2]{Hirschfeld3}, a line $L$ with Pl\"ucker coordinates $(p_{01},p_{02},p_{03},p_{12},p_{13},p_{23})$ intersects $T$ if and only if $s^4p_{23}+s^3tp_{02}+p_{03}s^2t^2+st^3p_{02}+t^4p_{01}=0$, or equivalently
    \[z_0s^4+z_1s^3t+z_2s^2t^2+z_3st^3+z_4t^4=f_L(s,t)=0.\] 
\end{proof}



\section{$G$-orbits of binary quartic forms} \label{three}
 We begin by fixing some notation that will be used throughout the rest of the paper:
\begin{description}
\item [$\overline{\bF_q}$] is  an algebraic closure of $\bF_q$,
\item[$G$] equals $PGL_2(q)$,
\item[$\phi$] is the Frobenius map $\phi(x)=x^q$ on $\overline{\bF_q}$,
\item[$\ep$] is a fixed quadratic non-residue in $\bF_q$,
\item[$\bF_q^\times$] is the multiplicative group of $\bF_q\setminus\{0\}$,
\item[$\gamma$] is  a generator of the cyclic group $\bF_q^\times$,
\item[$(\bF_q^\times)^2$] denotes the squares in $\bF_q^\times$ 
\item[Tr$_{\bF_{q^2}/\bF_q}(x)$] equals $x+x^q$ for $x \in \bF_{q^2}$,
\item[$N_{K/F} : K^\times \to F^\times $] is the norm function of a finite field extension $K/F$.
\end{description}
For a binary quartic form
\[ f(X,Y) =   z_0 X^4 +z_1 X^3Y+ z_2 X^2Y^2 + z_3 XY^3 +z_4 Y^4, \]
we recall (see \cite[Lemma 1.18(iii)]{Hirschfeld2}) that the discriminant of $f$ is 
\beq \label{eq:Delta_def} \Delta(f) = (z_0z_4-z_1z_3)^3+(z_0z_4-z_1z_3)^2 z_2^2 - z_2^3(z_0z_3^2+z_4z_1^2-z_0z_2z_4).\eeq
The discriminant $\Delta(f)$ vanishes if and only if $f(X,Y)$ has repeated factors. We say that $y/x \in \overline{\bF_q} \cup \{\infty\}$ is a root of $f$  if  $(Xy-Yx)$ is a factor of $f(X,Y)$ over $\overline{\bF_q}$. 
 \begin{definition} \label{Fi_def} 
 Let  $\mF^0$ denote the set  of forms with $\Delta(f)=0$. The forms with $\Delta(f) \neq 0$ decompose into the parts $\mF_1 \cup \mF_2 \cup \mF_2' \cup \mF_4 \cup \mF_4'$ as  defined in section \ref{statement}.
\end{definition}
We note the sizes of the parts in the decomposition above of $\bP(\mathrm{Sym}^4(V^*))$:
 \begin{enumerate} 
\item   $|\mF_{4}|=\tbinom{q+1}{4}=\tfrac{(q-2)|G|}{24}$ is the number of ways of picking $4$ distinct linear forms  over $\bF_q$.
\item $|\mF_4'|=\tbinom{ (q^2-q)/2 }{2}= \tfrac{(q-2)|G|}{8}$ is the number of ways to pick two distinct irreducible quadratic forms over $\bF_q$.
\item  $|\mF_{2}|=\tbinom{q+1}{2}\tfrac{q^2-q}{2}=\tfrac{q|G|}{4}$ is the number of ways of picking $2$ distinct linear forms  and an irreducible quadratic form over $\bF_q$.
\item $|\mF_2'|=\tfrac{q^4-q^2}{4}= \tfrac{q |G|}{4}$ is the number of irreducible quartic forms over $\bF_q$.
\item  $|\mF_{1}|=(q+1)\tfrac{q^3-q}{3}=\tfrac{(q+1)|G|}{3}$ is the number of ways of picking a linear form and an irreducible cubic form over $\bF_q$.
\item  From the above parts, we see that   $|\mF_1|+|\mF_2|+|\mF_2'|+|\mF_4|+|\mF_4'|$ equals $(q^4-q^2)$. Therefore, $|\mF^0|=(q^3+2q^2+q+1)$.
\end{enumerate}

\subsection{$G$-orbits of forms with $\Delta(f)=0$}
\begin{lem} \label{disc_zero_forms}
    The set $\mF^0$ of size $(1+q+2q^2+q^3)$  decomposes into the following $G$-orbits:
 \begin{enumerate}
\item $G \cdot X^4$ of size $(q+1)$,
\item $G \cdot X^3Y$ of size $q(q+1)$,
\item $G \cdot (X^2Y^2)$ of size $q(q+1)/2$,
\item  $G \cdot X^2Y(Y-X)$ of size $|G|/2$,
\item $G \cdot X^2(Y^2-\ep X^2)$ of size  $|G|/2$,
\item  $G \cdot (Y^2-\ep X^2)^2$ of size $q(q-1)/2$.
\end{enumerate}
\end{lem}
\begin{proof}
 The number of forms which split completely over $\bF_q$ and having 
\begin{enumerate}
    \item  a single factor of multiplicity $4$ is $(q+1)$,
 \item two  factors  with multiplicity $1$ and $3$, is $(q+1)q$,
 \item  two factors with multiplicity $2$ each, is $(q+1)q/2$,
 \item three factors with multiplicities $1,1$ and $2$ is $(q+1)q(q-1)/2$.
 \end{enumerate}
 Since $G$ acts triply transitively on $PG(1,q)$, it follows that each of these $4$ types of forms constitute a single orbit. Since $X^4, X^3Y, X^2Y^2, X^2Y(Y-X)$ represent these $4$ types of forms, they can be chosen as representatives of the $4$ orbits.
 
 The number of forms $f(X,Y)$ with exactly one factor of multiplicity $2$ over $\bF_q$ (and the  remaining two factors Galois conjugate over $\bF_{q^2}$) is $(q+1)(q^2-q)/2$. It is easy to see that the form $f(X,Y)=X^2(Y^2-\epsilon X^2)$ has stabilizer the group of order $2$ generated by $t \mapsto -t$: because $g(t)=c+dt$  stabilizes $f(X,Y)$ if and only if $g(t)= \pm t$, and any $g \in G$ fixing $f(X,Y)$ must fix the  linear form $X$ and hence $g$ is of the form $t \mapsto c+dt$. Therefore, $G \cdot X^2(Y^2-\epsilon X^2)$ is  the set of forms  with exactly one factor of multiplicity $2$ over $\bF_q$.

 The number of forms $f(X,Y)$ with exactly two Galois conjugate factors of multiplicity $2$ over $\bF_{q^2}$  is $(q^2-q)/2$. It is easy to see that the form $f(X,Y)=(Y^2-\epsilon X^2)^2$ has stabilizer the dihedral group of order $2(q+1)$ generated by the involution $t \mapsto -t$ and the cyclic group of order $(q+1)$ given by $\{t \mapsto \tfrac{\epsilon b+ at}{a+b t} \colon (a,b) \in PG(1,q)\}$ (isomorphic to $\bF_{q^2}^\times/\bF_q^\times$). Therefore, $G \cdot (Y^2-\epsilon X^2)^2$ is  the set of forms  with exactly two Galois conjugate factors of multiplicity $2$ over $\bF_q$. \end{proof}

 \subsection{Cross-ratio and $\jmath$-invariant} Let  $F$ be an arbitrary field of characteristic $3$. We will identify   $\bP^1(F)$  with the projective line $F \cup \{\infty \}$ where $(s,t) \in \bP^1(F)$ is identified with $t/s \in F \cup \{ \infty\}$. We recall that  the cross-ratio $(P_1, P_2;P_3,P_4)$ of $4$ distinct and ordered points $(P_1, P_2, P_3, P_4)$ of $\bP^1(F)$ is
\beq \label{eq:crossratiodef} \lambda = (P_1, P_2;P_3,P_4)=  \frac{(P_3-P_1)(P_4-P_2)}{(P_3-P_2)(P_4-P_1)}. \eeq
We need two standard lemmas (see  \cite[\S 6.1, \S 1.11]{Hirschfeld2}). The \emph{anharmonic group} $G_*$ is the subgroup of $PGL_2(F)$ given by 
\[G_*=\{ t \mapsto t^{\pm 1}, 1 - t^{\pm 1}, 1/(1-t^{\pm 1}) \}. \]
It is generated by the order $3$ element  $t \mapsto 1/(1-t)$ and the involution  $t \mapsto t^{-1}$, and is isomorphic to the symmetric group $S_3$.

\begin{definition}
        Let  $\sigma_1, \dots, \sigma_4$ denote the following elements of the symmetric group $S_4$:
    \[ \sigma_1=(12)(34), \; \sigma_2=(13)(24),\;  \sigma_3=(12), \text{ and }\,  \sigma_4=(234).\]
\end{definition}

\begin{lem} \label{cross} The cross-ratio  $(P_1, P_2;P_3,P_4)$ gives a bijection between $F\setminus\{0,1\}$ and the set of  $PGL_2(F)$-orbits of ordered $4$-tuples $(P_1, \dots, P_4)$ of distinct points of the projective line over $F$. \\
The set of cross-ratios associated to the unordered set $\{P_1, P_2, P_3, P_4\}$  consists of the orbit of $\lambda = (P_1, P_2;P_3,P_4) $  under the  anharmonic group $G_*$. 
The homomorphism 
\[ \rho:S_4 \to G_*, \quad \rho(\sigma)( (P_1, P_2;P_3,P_4) )= (P_{\sigma(1)}, P_{\sigma(2)};P_{\sigma(3)},P_{\sigma(4)}), \]
 has kernel, the normal subgroup $\langle (12)(34), (13)(24)  \rangle \simeq \bZ/2 \bZ  \times \bZ/2 \bZ$. The quotient group $S_4/\mathrm{ker}(\rho)$  is isomorphic to $G_*$, and $\rho$ sends  the transposition $(12)$ and the $3$-cycle $(234)$  to  the involution  $\lambda \mapsto \lambda^{-1}$ and the order $3$ element  $\lambda \mapsto 1/(1-\lambda)$, respectively.
\end{lem}

The $\jmath$-invariant of an unordered set $\{P_1, P_2, P_3, P_4\}$ is given by:
\[ \jmath(\lambda)= \frac{(\lambda+1)^6}{\lambda^2(\lambda-1)^2}, \]
where $\lambda$ is the cross-ratio of any ordering of $\{P_1, \dots, P_4\}$.

\begin{lem} \label{jmath_lem}
The function $\jmath(\lambda)$ gives a bijection between the image $\jmath(F\setminus\{0,1\})$ and  the set of $G_*$-orbits on $F \setminus \{0,1\}$.  Each of these $G_*$ orbits has size $6$ with one   exception $ \jmath^{-1}(0) = \{-1\}$.
\end{lem}

\begin{lem} \label{stabilizer_lemma} 
Let $f$ be a binary quartic form over $\bF_q$ with $\Delta(f) \neq 0$. Let $F$ denote the splitting field of $f$. Let Stab$_F(f)$ denote the stabilizer of $f$ in $PGL_2(F)$ and let Stab$(f)$ denote the stabilizer of $f$ in $G = PGL_2(q)$.  The permutation representation of $\mathrm{ Stab}_F(f)$ on the roots $( (s_1,t_1),\dots, (s_4,t_4) )$ gives an injective homomorphism $\sigma: \mathrm{ Stab}_F(f) \to S_4$ mapping $g \mapsto \sigma_g$. The group $\mathrm{Stab}_F(f)$ is isomorphic via this homomorphism to: 
\begin{enumerate}
\item $ \langle (12)(34), (13)(24) \rangle \simeq  \bZ/2 \bZ  \times \bZ /2 \bZ\,$ if $\jmath(f) \neq 0$,
\item $S_4\,$ if $\jmath(f) =0$.
\end{enumerate}
\end{lem}

\begin{proof}
Let $((s_1,t_1), \dots, (s_4,t_4))$ be an ordering of the roots of $f$, and let $\lambda_0$ denote the cross-ratio of this ordering. 
For $g \in  \text{Stab}_F(f)$, let $\sigma_g \in S_4$ be defined by  $g \cdot ((s_1,t_1), \dots, (s_4,t_4)) = ((s_{\sigma(1)},t_{\sigma(1)}), \dots, (s_{\sigma(4)},t_{\sigma(4)}))$. The homomorphism  $\sigma:\text{Stab}_F(f) \to S_4$ is injective, because any $g \in PGL_2(F)$ fixing three distinct points is the identity element.

The involutions  $\sigma_1=(12)(34)$ and $\sigma_2=(13)(24)$ are always in $\sigma(\text{Stab}_F(f))$.  To see this, let $h_1$ be the unique element of $PGL_2(F)$ sending $(r_1, r_2, r_3)$ to $(r_2, r_1, r_4)$. The following  cross-ratios are equal
\[ (r_1,r_2;r_3,r_4)=\lambda_0 = (r_2, r_1; r_4, h_1(r_4)) =(r_1, r_2; h_1(r_4), r_4).\] 
Comparing the first and last tuples, we conclude $h_1(r_4)=r_3$. Thus, $\sigma_1=\sigma_{h_1}$.
 
 Similarly, let $h_2$ be the unique element of $PGL_2(F)$ sending $(r_1, r_2, r_3)$ to $(r_3, r_4, r_1)$. The following  cross-ratios are equal
\[ (r_1,r_2;r_3,r_4)=\lambda_0 = (r_3, r_4; r_1, h_2(r_4)) =(r_1, h_2(r_4); r_3, r_4).\] 
Comparing the first and last tuples, we conclude $h_2(r_4)=r_2$. Thus, $\sigma_2=\sigma_{h_2}$.
 
 The  quotient group $S_4/\langle \sigma_1, \sigma_2 \rangle \simeq S_3$ has  $3$ elements of order $2$, namely the cosets represented by $(12), (13), (14)$ and two elements of order $3$ represented by $\sigma_4$ and $\sigma_4^2$.  
 Let $h_3$  be the unique element of $PGL_2(F)$ sending $(r_1, r_2, r_3)$ to $(r_2, r_1, r_3)$. By the properties of cross-ratio,  $\lambda_0=(r_2, r_1; r_3, h_3(r_4))$ whereas $(r_1, r_2; r_3, h_3(r_4))=\lambda_0^{-1}$. Therefore, $h_3(r_4)=r_4$ if and only if $\lambda_0=-1$. Thus, $\sigma_3=\sigma_{h_3}$ is in $\sigma(\text{Stab}_F(f))$ if and only if  $\lambda_0=-1$. An identical argument with $(12)$ replaced by $(13)$ and $(14)$, shows that $(13)  \in  \sigma(\text{Stab}_F(f))$ if and only if  $\lambda_0=-1$, and 
$(14) \in \sigma(\text{Stab}_F(f))$ if and only if  $\lambda_0=-1$. Let $h_4$  be the unique element of $PGL_2(F)$ sending $(r_2, r_3, r_4)$ to $(r_3, r_4, r_2)$. By the properties of cross-ratio,  
\[ \lambda_0 =(h_4(r_1), r_3; r_4, r_2), \quad (h_4(r_1), r_2; r_3, r_4) = 1 - \lambda_0^{-1}.\]
Therefore, $h_4(r_1)=r_1$ if and only if $\lambda_0 =1 - \lambda_0^{-1}$ which is equivalent to $\lambda_0=-1$. Thus $\sigma_4=\sigma_{h_4}$ is in $\sigma(\text{Stab}_F(f))$ if and only if  $\lambda_0=-1$. 

This completes the proof that Stab$_F(f) \simeq S_4$ if $\jmath(f)=0$, and Stab$_F(f) = \langle (12)(34), (13)(24) \rangle \simeq \bZ/2 \bZ  \times \bZ /2 \bZ$ if $\jmath(f) \neq 0$. 
 \end{proof}

 \subsection{Cross-ratios of restricted orderings of roots of quartic forms}
For a quartic form $f$ over $\bF_q$ with $\Delta(f) \neq 0$, it will be useful to consider  orderings of the roots of $f$ which satisfy certain restrictions:
 \begin{definition} \label{restricted_ordering}
For $f  \in \mF$ where $\mF\in \{\mF_4', \mF_2', \mF_{2}, \mF_{1}\}$,  a \emph{restricted ordering} of the roots of $f$ is an ordering of the form:
 \begin{description}
 \item  [$\mF=\mF_{4}:$ ]\emph{ no restriction on the ordering $(r_1, r_2, r_3, r_4)$}.
\item[$\mF=\mF_4':$ ]  $(r_1, r_2, r_3, r_4)=(r_1, \phi(r_1), r_3, \phi(r_3))$.
\item  [$\mF= \mF_{2}:$ ]  $(r_1, r_2, r_3, r_4)=(r_1,r_2, r_3, \phi(r_3))$\emph{ with }$r_1, r_2 \in PG(1,q)$.
\item[$\mF= \mF_2':$ ]   $(r_1, r_2, r_3, r_4)=(r_1, \phi^2(r_1), \phi(r_1), \phi^3(r_1))$.
\item  [$\mF=\mF_1:$ ]   $(r_1, r_2, r_3, r_4)=(r_1,r_2,\phi(r_2),\phi^2(r_2))$ \emph{ with } $r_1 \in PG(1,q)$.
\end{description}
Let $\text{ord}(f)$ denote the set of all restricted orderings of the roots of $f$. \\
We note that $g\in G$ carries Ord$(f)$ to Ord$(g \cdot f)$.
\end{definition}

We  define some subgroups of $S_4$:
\beq
 \label{eq:ZmF}
 Z(\mF)= \begin{cases} S_4 &\text{ if $\mF=\mF_4$}\\
 \langle (234) \rangle \simeq \bZ/3 \bZ &\text{ if $\mF=\mF_1$}\\ 
  \langle (12), (34) \rangle \simeq \bZ/2 \bZ  \times \bZ /2 \bZ &\text{ if $\mF=\mF_2$}\\
\langle (12), (1324) \rangle \simeq D_4 &\text{ if $\mF=\mF_4'$}\\
 \langle (1324) \rangle \simeq  \bZ/4 \bZ &\text{ if $\mF=\mF_2'$}\end{cases}
\eeq

The Frobenius map $\phi$ acts Ord$(f)$  by $\phi(r_1, r_2, r_3, r_4)=(\phi(r_1), \dots, \phi(r_4))$. Clearly  $(\phi(r_1), \dots, \phi(r_4))$ is in Ord$(f)$ and also a  permutation  $(r_{\sigma_\phi(1)}, \dots, r_{\sigma_\phi(4)})$ of $(r_1, \dots, r_4)$.
  \begin{lem}  \label{ZS4}
 Let $f \in \mF$ where $\mF \in \{\mF_4, \mF_2, \mF_1, \mF_4', \mF_2'\}$. The permutation  $\sigma_{\phi}(\mF) \in S_4$  above, associated to the action of the Frobenius map $\phi$ on Ord$(f)$ is given by:
 \[ \sigma_{\phi}(\mF) = \begin{cases}  \mathrm{identity} &\text{if $\mF=\mF_4$},\\ 
 (12)(34) &\text{ if $\mF=\mF_4'$,}\\
   (34) &\text{ if $\mF=\mF_2$,}\\
      (1324) &\text{ if $\mF= \mF_2'$,}\\
   (234) &\text{ if $\mF=\mF_1$.} \end{cases}\]

The centralizer $Z_{S_4}(\sigma_\phi(\mF))$ of $\sigma_\phi$ in $S_4$ equals the group  $Z(\mF)$  given in  \eqref{eq:ZmF}.

The set  $\mathrm{ord}(f)$ forms a single orbit  $Z(\mF)$-orbit. 
\end{lem}
\begin{proof}
If $\mF=\mF_4$, then $\sigma_\phi$ acts trivially on Ord$(f)$ and hence $Z_{S_4}(\sigma_\phi)=S_4$. If $\mF=\mF_4'$, then  $\sigma_\phi$ acts on $(r_1, \phi(r_1), r_2, \phi(r_2)$ by the permutation  $(12)(34)$. If $\mF=\mF_2$, then $\sigma_\phi$ acts on $(r_1, r_2, r_3, \phi(r_3))$ by the permutation  $(34)$. If $\mF=\mF_2'$, then $\sigma_\phi$ acts on $(r_1, \phi^2(r_1), \phi(r_1), \phi^3(r_1)$ by the permutation  $(1324)$. If $\mF=\mF_1$, then $\sigma_\phi$ acts on $(r_1, r_2, \phi(r_2),\phi^2(r_2)$ by the permutation  $(234)$.\\

It is easy to check that stabilizers of identity, $(12)(34)$, $(34)$, $(1324)$, and  $(234)$ are $S_4$, $\langle (12), (1324) \rangle$,  $\langle (12), (34) \rangle$, 
$\langle(1324)\rangle$, and 
$\langle(234)\rangle$ respectively.  In other words $Z_{S_4}(\sigma_\phi)=Z(\mF)$ as defined in \eqref{eq:ZmF}.

If  $(r_1, \dots, r_4)$ and $(r_1', \dots, r_4')$  are in $ \text{ord}(f)$, let $\sigma \in S_4$ such that  $(r_1',\dots,r_4') = (r_{\sigma(1)}, \dots, r_{\sigma(4)})$.
It follows that 
\[ r_{\sigma_\phi \sigma(i)}= \phi(r_{\sigma(i)})  =\phi(r_i')= r_{\sigma_{\phi}(i)}'
= r_{\sigma \sigma_{\phi}(i)}
,\]
and hence $\sigma \in Z_{S_4}(\sigma_\phi)$.
\end{proof} 

\begin{prop} \label{stabilizer_prop} Given $f$ with $\Delta(f) \neq 0$, let $(r_1,r_2,r_3,r_4)$ be a fixed restricted ordering of the roots of $f$, and let 
 $\lambda_f$ denote the cross-ratio of this ordering. Let $F$ be a splitting field of $f$ over $\bF_q$. We recall from  Lemma \ref{stabilizer_lemma} the injective homomorphism $\sigma:\text{Stab}_F(f) \hookrightarrow  S_4$. We identify $\text{Stab}_F(f)$ with its image in $S_4$. Let $\sigma_{\phi} \in S_4$ be as in Lemma \ref{ZS4}.  The  stabilizer group Stab$(f) =\mathrm{Stab}_F(f) \cap PGL_2(q)$ equals the centralizer of $\sigma_\phi$ in $\mathrm{Stab}_F(f)$
\[ \mathrm{Stab}(f)=Z_{\mathrm{Stab}_F(f)}(\sigma_\phi), \] 
and is isomorphic to:\\

In case $\jmath(f) \neq 0$:
\begin{enumerate}
\item[(1)]$  \langle (12)(34), (13)(24) \rangle \simeq \bZ/2 \bZ \times \bZ/2 \bZ$ if  $f \in \mF_4$,
\item[(2)]$  \langle (12)(34), (13)(24) \rangle \simeq \bZ/2 \bZ \times \bZ/2 \bZ$ if  $f \in   \mF_4'$,
\item[(3)]   $\langle (12)(34) \rangle \simeq  \bZ/2 \bZ$  if  $f \in \mF_2$,
\item[(4)]  $\langle (12)(34) \rangle \simeq  \bZ/2 \bZ$ 
 if  $f \in  \mF_2'$,
\item[(5)] \emph{the trivial group}  if  $f \in \mF_1$.\\
\end{enumerate}

In case  $\jmath(f) =0$: 
\begin{enumerate}
\item[(a)] $S_4$ if 
$f \in \mF_4$,
\item[(b)] 
$ \langle (12), (1324)  \rangle \simeq D_4$
 if  $f \in \mF_4'$,
\item[(c)] $\langle (12),(34)\rangle   \simeq \bZ/2 \bZ \times \bZ/2 \bZ$
  if  $f \in \mF_2$,
\item[(d)] $\langle (1324) \rangle \simeq \bZ/4 \bZ$    if  $f \in \mF_2'$,
\item[(e)] $\langle (234) \rangle \simeq \bZ/3 \bZ $  if  $f \in \mF_1$.
\end{enumerate}
\end{prop}

\begin{proof}
For $g \in \text{Stab}_F(f)$ to be in $G$ we must have $\phi(g)=g$ which is equivalent to $\sigma_{\phi(g)}=\sigma_g$. Since $\sigma_{\phi(g)} = \sigma_\phi \sigma_g (\sigma_\phi)^{-1}$ we get
\[ \text{Stab}(f) = \{ g \in \text{Stab}_F(f) : \sigma_\phi \sigma_g (\sigma_\phi)^{-1}= \sigma_g\} =Z_{\text{Stab}_F(f)}(\sigma_\phi). \]
Therefore, $\text{Stab}(f) = \text{Stab}_F(f)\cap Z_{S_4}(\sigma_\phi)$.

If $f \in \mF_4$ then $F = \bF_q$ and hence $\text{Stab}(f) = \text{Stab}_F(f)$ is as given in  Lemma \ref{stabilizer_lemma}:
\begin{description}
\item[--]  $ \langle (12)(34), (13)(24) \rangle$ if $\jmath(f) \neq 0$,
\item[--]  $S_4$ if  $\jmath(f)=0$.
\end{description}

This proves the  assertions about $\mF_4$ in $1),  a)$.

 If $f \in \mF_4'$ then Stab$(f)=\text{Stab}_F(f) \cap \langle (12), (1324) \rangle$. Thus, Stab$(f)$ equals
\begin{description}
\item[--]  $\langle (12), (1324) \rangle$ if  $\jmath(f)=0$,
\item[--]  $ \langle (12)(34), (13)(24) \rangle$  if $\jmath(f) \neq 0$.
\end{description}
This proves assertions 2), b). If $f\in \mF_2$ then Stab$(f)=\text{Stab}_F(f) \cap \langle (12), (34) \rangle$. Thus, Stab$(f)$ equals
\begin{description}
\item[--]  $ \langle (12), (34) \rangle$ if $\jmath(f)=0$, 
\item[--]  $ \langle (12)(34) \rangle$ if $ \jmath(f) \neq 0$.
\end{description}
This proves assertions 3), c). If $f\in \mF_2'$ then Stab$(f)=\text{Stab}_F(f) \cap 
\langle (1324)\rangle$.  Thus, Stab$(f)$ equals
\begin{description}
\item[--]  $ \langle (1324) \rangle$ if $\jmath(f)=0$, 
\item[--]  $ \langle (12)(34) \rangle$ if $ \jmath(f) \neq 0$.
\end{description}
This proves assertions 4), d). If $f\in \mF_1$ then Stab$(f)=\text{Stab}_F(f) \cap 
\langle (234)  \rangle$.  Thus Stab$(f)$ equals
\begin{description}
\item[--]  $ \langle (234) \rangle$ if $\jmath(f)=0$, 
\item[--]  trivial  if $ \jmath(f) \neq 0$.
\end{description}
This proves assertions 5), e).
 \end{proof}


\begin{definition}
Let $\rho:S_4 \to G_*$ be the homomorphism defined in Lemma \ref{cross}. 
The subgroup  $\rho(Z_{S_4}(\sigma_\phi(\mF))$ of the anharmonic group $G_*$ will be denoted by $H_i$ in the case of $\mF=\mF_i, i=1,2,4$  and $H_i'$ in the case $\mF=\mF_i', i=2,4$. 
\end{definition}
\begin{prop}
The groups $H_i, H_i'$ in the definition above equal:
\begin{align}
\nonumber H_1&=\langle t \mapsto  1/(1-t) \rangle, \\
H_2&=\langle t \mapsto  1/t\rangle, \\
\nonumber H_4&=G_*, \\
\nonumber H_4'&=H_2'=H_2.
\end{align}
\end{prop}
\begin{proof} The groups $H_i, H_i'$ are calculated as follows:  From Lemma \ref{cross}, we know  that  i) ker$(\rho) = \langle \sigma_1, \sigma_2 \rangle$, ii)$\rho(\sigma_3)$ is the map $t \mapsto t^{-1}$, and iii) $\rho(\sigma_4)$ is the map $t \mapsto 1/(1-t)$. Using this in \eqref{eq:ZmF}, we get $\rho(Z_{S_4}(\sigma_\phi))$ equals i) $G_*$ in case of $\mF_4$, ii) $H_1$ in case of $\mF_1$, and iii) $H_2$ in case of $\mF_4', \mF_2, \mF_2'$.\end{proof} 

We define some subsets $ \tilde \mN_i$ and  $\mN_i$  of $\overline{\bF_q}$. We will prove in Proposition \ref{Lambda_prop} below that the cross-ratio of a restricted ordering of the roots of $f \in \mF_i, \mF_i'$ will lie in $\tilde \mN_i$ where
\begin{align} \label{eq:Ni_def}
\nonumber  \tilde \mN_4&=  \bF_q\setminus\{0,1\}, \\
\tilde \mN_2&= \{\lambda \in \bF_{q^2}\setminus\{1\} : N_{\bF_{q^2}/\bF_q}(\lambda)=1 \},\\
\nonumber \tilde \mN_1&=\{\lambda \in \bF_{q^3} : \lambda^{q+1}- \lambda^q+1=0 \},\\
\nonumber \mN_i&=\tilde \mN_i \setminus\{ -1\}. 
\end{align}
We note that $|\tilde \mN_i|=(q+2-i)$: this is obvious for $i=2, 4$, and for $i=1$, we observe that for $\bF_{q^3}=\bF_q[\theta]$, the $(q+1)$ values $\lambda_x= (x, \theta; \phi(\theta), \phi^2(\theta) )$ for  $x \in \bF_q \cup \{\infty\}$ are in $\mN_1$ because, 
\[ (\lambda_x)^q =(x,\phi(\theta);  \phi^2(\theta), \theta) =\frac{1}{1-\lambda_x},\] 
and the condition $\lambda^q = 1/(1-\lambda)$ is the same as $\lambda^{q+1} - \lambda^q+1=0$. 
Therefore, $|\mN_i|=q+1-i$.\\

\begin{lem} \label{Ji_def}
The function $\jmath(\lambda)$ induces injective maps $\bar \jmath:  \mN_i/H_i \to \bF_q$. 
 The sets   $J_i = \jmath(\mN_i), i=1,2,4$ partition $\bF_q^\times$ and have sizes:
 \beq  \label{eq:Ji_def} |J_4|=\tfrac{q-3}{6}, \quad |J_2|=\tfrac{q-1}{2}, \quad  |J_1|=\tfrac{q}{3}.  \eeq

\begin{enumerate}
\item The  map $\bar \jmath:  \mN_i/H_i \to J_i$ gives a bijection between $H_i$-orbits on $\mN_i$ and $J_i$.
\item  The  map $\bar \jmath:  \mN_4/H_2 \to J_4$ is surjective and $3$-to-$1$, and hence the number of $H_4'=H_2$-orbits on $\mN_4$ is $3 |J_4|$.
	 \end{enumerate}
\end{lem}
\begin{proof}
Since $\jmath(\lambda)$ is constant on $G_*$-orbits and $H_i \subset G_*$,  we get induced maps $\bar \jmath:  \mN_i/H_i \to \overline{\bF_q}$. We first show that $J_1, J_2, J_4$ are disjoint subsets of $\bF_q$:
 We note that the sets $\mN_1 \subset \bF_{q^3}\setminus \bF_q$, $\mN_2 \subset \bF_{q^2}\setminus \bF_q$ and $\mN_4 \subset  \bF_q$. Therefore, if  $\lambda_i \in \mN_i$, $\lambda_j \in \mN_j$ with $i \neq j$,     the $G_*$-orbits  $G_*(\lambda_i)$ and $G_*(\lambda_j)$ are distinct. Thus, $J_1, J_2, J_4$ are disjoint sets of $\overline{\bF_q}$. Clearly $\mN_4 \subset \bF_q$, hence $J_4 \subset \bF_q^\times$.
For $\lambda \in \mN_1$, we can write the expression for $\jmath(\lambda)$ as 
$\jmath(\lambda)=  N_{\bF_{q^3}/\bF_q}(\lambda-\phi(\lambda) )
$  which shows that $J_1 \subset \bF_q \setminus \{0,1728\}$. For $\lambda \in \mN_2$ we can write the expression for $\jmath(\lambda)$ as 
\[ \jmath(\lambda)=\tfrac{(\phi(\lambda) +\lambda-1)^3}{\phi(\lambda)+\lambda +1}
,\]
which shows that $J_2 \subset \bF_q^\times$.\\

Next we show that $\bar \jmath:\mN_i/H_i \to \bF_q$  is injective. If $i=4$, then $H_4=G_*$ and hence $\bar \jmath$ is injective. We now assume $\{i,j\}=\{1,2\}$. 
Suppose  $g \in G_*$ and $\lambda \in \mN_i$ such that $g(\lambda) \in \mN_i$, then we must show 
$g(\lambda)=h(\lambda)$ for some $h \in H_i$.  Since $G_* = H_iH_j$, we may assume $g \in H_j$. If $g$ is the identity element of $H_j$ then we take $h$ to be the identity element of $H_i$.  So we assume $g$ is a nontrivial element of $H_j$. If $i=1$ then $g(\lambda)=1/\lambda$ does not satisfy  $x^{q+1}- x^q+1=0$  unless $q \equiv 1 \mod 3$ and $\lambda=-1 $ which is not in $\mN_1$. If $i=2$ then $g(\lambda)$ is  $1/(1-\lambda)$ or $\lambda/(\lambda-1)$ which do not satisfy the equation $N_{\bF_{q^2}/\bF_q}(x)=1$ 
 unless  $\lambda =-1$ which is not in $\mN_2$.\\

Finally, we show that $J_1,J_2,J_4$ partition $\bF_q^\times$. As shown above, the function  $\jmath(\lambda)$ is  injective on $\mN_i/H_i$. Therefore $|J_i|=|\mN_i/H_i|$.  Since $\mN_4$ has size $(q-3)$ and each $H_4=G_*$ orbit has size $6$ we see that $ |J_4|=\tfrac{q-3}{6}$. Similarly, $|\mN_2| = (q-1)$,   $|\mN_1| = q$,  each $H_2$ orbit on $\mN_2$ has size $2$,  and each $H_1$-orbit on  $\mN_1$ orbit has size $3$. Therefore, $|J_2|=|\mN_2/H_2| =(q-1)/2$ and   $|J_1|=|\mN_1/H_1| =q/3$.
 Since $|J_4|+|J_2|+|J_1|=(q-1)$ we see that the sets $J_4, J_2,J_1$ partition  $\bF_q^\times$. This completes the proof of part (1) of the lemma. \\
 Since $\mN_4/H_2 \to \mN_4/H_4$ is a $3$-to-$1$ map, it follows that $\mN_4/H_2 \to J_4$ is a $3$-to-$1$ surjective map. This proves part (2).
\end{proof}

\subsection{$G$-orbits of forms with zero $\jmath$-invariant}  \label{j_zero_forms}
\begin{prop} \label{j_zero_forms_prop}
The $\jmath$-invariant of a form $f(X,Y)= z_0 X^4+z_1 X^3Y+ \dots+ z_4 Y^4$ with $\Delta(f) \neq 0$ equals $0$  if and only if  $z_2=0$. The size of the set of such  forms is $(q^3-q)=|G|$ and it decomposes into the following $G$-orbits:
\begin{enumerate}
    \item $G\cdot XY(Y^2-X^2)$ in $\mF_4$ of size $|G|/24$
    \item $G\cdot XY(Y^2-\ep X^2)$ in $\mF_2$ of size $|G|/4$.
    \item $G\cdot X(Y^3-YX^2 -r_0 X^3)$ in $\mF_1$ of size $|G|/3$. Here Tr$_{\bF_q/\bF_3}(r_0) \neq 0$.
    \item $G \cdot (Y^2-\epsilon X^2)(Y-\alpha X)(Y - \alpha^q X)$ in $\mF_4'$  has size $|G|/8$. Here  $\alpha \alpha^q=\epsilon$.
\item      $G \cdot  f(X,Y)$ in $\mF_2'$ of size $|G|/4$ where \[f(X,Y)=\begin{cases} 
G\cdot (Y^4-\gamma X^4) &\text{ if $q \equiv 1 \mod 4$},\\
G\cdot (Y^4-X^4 +X^3Y) &\text{ if $q \equiv 3 \mod 4$.}\end{cases}
\]
   
\end{enumerate}
\end{prop}
\begin{proof} Since $\jmath(f)  = z_2(f)^6/\Delta(f)$, it follows that $\jmath(f)=0$ if and only if $z_2(f)=0$. 
The number of forms with $z_2(f)=0$ is clearly $(1+q+q^2+q^3)$. From \eqref{eq:Delta_def}, we see that a form $f$ with $z_2(f)=0$ satisfies $\Delta(f)=0$ if and only if $z_0z_3-z_1z_2 = 0$. Thus, the number of forms with $z_2(f)=\Delta(f)=0$ is $(1+q)^2$, and hence, the number of forms in $\mF\setminus \mF^0$ with $\jmath(f)=0$ is $(1+q+q^2+q^3) - (1+q)^2=q^3-q$.\\

The roots of $XY(Y^2-X^2)$ and $XY(Y^2-\ep X^2)$ are  $\infty, 0, 1,-1$ and $\infty, 0, \sqrt\ep,-\sqrt\ep$ respectively, and hence these forms are in $\mF_4$ and $\mF_2$ respectively.  The polynomial $T^3-T-r_0$ has no roots in $\bF_q$ when Tr$_{\bF_q/\bF_3}(r_0) \neq 0$ (by \cite[Corollary 1.23]{Hirschfeld2}) and hence $X(Y^3-YX^2 -r_0 X^3)$ is in $\mF_1$. All the four roots of $(Y^2-\epsilon X^2)(Y-\alpha X)(Y - \alpha^q X)$ are in $\bF_{q^2}\setminus\bF_q$ and hence this form is in $\mF_4'$. Finally, if $q \equiv 1 \mod 4$ then all the roots $\pm \gamma^{1/4}, \pm \imath \gamma^{1/4}$ of $Y^4 -\gamma X^4$ are in $\bF_{q^4}\setminus\bF_{q^2}$ and hence this form is in $\mF_2'$.
If $q \equiv 3 \mod 4$, then  the roots of $f(X,Y)=Y^4 - X^4 +X^3Y$ are 
 $\imath + (1+\imath)\sqrt{\imath-1}$ for each of the two choices of $\imath$, and each of the two choices of $\sqrt{\imath-1}$, and hence the splitting field of $f(X,Y)$ is $\bF_{q^4}$, i.e. $f\in \mF_2'$. ( We note that $N_{\bF_{q^2}/\bF_q}(\imath-1)=-1$ is a non-square in $\bF_q$ and hence $(\imath-1)$ is a non-square in $\bF_{q^2}$.)\\

We note that the coefficient of $X^2Y^2$ in all the five  representative forms in the list above are zero, and hence these forms have zero $\jmath$-invariant.  They represent distinct orbits, as they are in  distinct  types $\mF_4,\mF_2,\mF_1, \mF_4', \mF_2'$. The sizes of these orbits follow from the size of the stabilizers given in Proposition \ref{stabilizer_prop}. Finally, the sizes of these orbits add up to $|G|$, and hence these orbits exhaust the set of forms with $\jmath$-invariant zero.
\end{proof}

\subsection{$G$-orbits of forms with nonzero $\jmath$-invariant}  \label{j_nonzero_forms}
As noted after Definition \ref{Fi_def}, the sizes of $\mF_4, \mF_4', \mF_2, \mF_2'$ and $\mF_1$ are $\tfrac{(q-2)|G|}{24}$,  $\tfrac{(q-2)|G|}{8}$,  $\tfrac{q|G|}{4}$,   $\tfrac{q |G|}{4}$, and   $\tfrac{(q+1)|G|}{3}$, respectively. The size of the subset of forms of $\jmath$-invariant zero, in these five types of forms are $|G|/24$, $|G|/8$, $|G|/4$, $|G|/4$ and $|G|/3$, respectively, as shown in Proposition \ref{j_zero_forms}. Thus, the number of forms with non-zero $\jmath$-invariant in these five types are: $\tfrac{(q-3)|G|}{24}$,  $\tfrac{(q-3)|G|}{8}$, $\tfrac{(q-1)|G|}{4}$, $\tfrac{(q-1)|G|}{4}$, and $\tfrac{q|G|}{3}$, respectively. The stabilizer of a form in each of these five types has size $4,4,2,2$, and $1$, respectively, as calculated in Proposition \ref{stabilizer_prop}. Therefore, the number of orbits of forms with non-zero $\jmath$-invariant in $\mF_4, \mF_4', \mF_2, \mF_2'$ and $\mF_1$ are
$\tfrac{(q-3)}{6}$,  $\tfrac{(q-3)}{2}$, $\tfrac{(q-1)}{2}$, $\tfrac{(q-1)}{2}$, and $\tfrac{q}{3}$, respectively.\\

\begin{prop} \label{Lambda_prop}
Given $f \in \mF_i$  the set $\Lambda_f$ of cross-ratios of restricted orderings $\mathrm{ord}(f)$ of the roots of $f$,  forms a single $H_i$-orbit   in $\tilde \mN_i$. Similarly, for $f \in \mF_i'$ with $i\in \{1,2\}$, the set $\Lambda_f$ of cross-ratios of restricted orderings $\mathrm{ord}(f)$ of the roots of $f$,  forms a single $H_i'$-orbit   in $\tilde \mN_i$.  \\
Let $\mF_i/G$  (resp. $\mF_i'/G$) denote the set of $G$-orbits on $\mF_i$  (resp. $\mF_i'$).  The map that takes $f \in \mF_i$ (resp. $f \in \mF_i'$) to   $\Lambda_f$   in  $\tilde \mN_i$ induces a bijective map
\begin{align*} 
\mF_i/G  &\leftrightarrow \tilde\mN_i/H_i, \quad i=1,2,4.\\
\mF_i'/G  &\leftrightarrow \tilde\mN_i/H_i',  \quad i=2,4.
  \end{align*}
\end{prop}
\bep
Let $f \in \mF_i$ (resp. $\mF_i'$). First, we show that $\Lambda_f$ forms a single $H_i$-orbit  (resp $H_i'$-orbit)  in $\overline{\bF_q}$. We consider the map ord$(f) \to \Lambda_f$ that takes a restricted ordering $(r_1, r_2, r_3, r_4)$ of the roots of $f$ to its cross-ratio $\lambda_f=(r_1, r_2; r_3, r_4)$. As shown in Lemma \ref{ZS4}, the set ord$(f)$ is the orbit  $Z_{S_4}(\sigma_\phi(\mF)) \cdot (r_1, r_2, r_3, r_4)$ and $\rho(Z_{S_4}(\sigma_\phi(\mF)))$ equals $H_i$ if $\mF \in \{\mF_1, \mF_2, \mF_4\}$ 
and $H_i'$ if $\mF \in \{\mF_2', \mF_4'\}$. Thus, $\Lambda_f = H_i \cdot \lambda_f$ if $f \in \mF_i$ and $\Lambda_f = H_i' \cdot \lambda_f$  if $f \in \mF_i'$.\\
 Next we show that $\lambda_f \in \tilde \mN_i$ for $f \in \mF_i \cup \mF_i'$.
\begin{enumerate}
\item If $f \in \mF_4$ then clearly  $\lambda_f \in \bF_q \setminus\{0,1\}=\tilde \mN_4$.
\item  If  $f\in \mF_4'$ then   $\lambda_f=N_{\bF_{q^2}/\bF_q}(\beta)$ where $\beta = \tfrac{r_3-r_1}{r_3-\phi(r_1)}$. Hence, $\lambda_f \in  \tilde \mN_4$. 
\item If  $f\in \mF_{2}$  then  $\lambda_f=\tfrac{\beta}{\phi(\beta)}$ where $\beta=(r_3-r_1)( \phi(r_3)- r_2)$. Hence, $\lambda_f\in  \tilde \mN_2$. 
\item If  $f\in \mF_2'$  then $\lambda_f=\tfrac{\beta}{\phi(\beta)}$ where $\beta=N_{\bF_{q^4}/\bF_{q^2}}(\phi(r_1)-r_1)$. Hence, $\lambda_f\in  \tilde \mN_2$. 
\item If  $ f \in\mF_{1}$  then  $\lambda_f=\tfrac{-\beta}{\phi(\beta)}$ where $\beta=(\phi(r_2)-r_1)(\phi^2(r_2)-r_2)$ satisfies 
\[\lambda_f^q(1-\lambda_f)=\tfrac{\beta+\phi(\beta)}{-\phi^2(\beta)}=\tfrac{r_2(\phi^2(r_2)-\phi(r_2)) -r_1(\phi^2(r_2)-\phi(r_2))}{r_2(\phi^2(r_2)-\phi(r_2)) -r_1(\phi^2(r_2)-\phi(r_2))}=1.\]
Thus, $\lambda_f \in  \tilde \mN_1$.
\end{enumerate}

Next we show that for $f \in \mF_i$, the map $f \mapsto H_i \cdot \lambda_f$ induces an injective map  $\mF_i/G\to \tilde \mN_i/H_i$ and similarly for $f \in \mF_i'$, the map $f \mapsto H_i' \cdot \lambda_f$ induces an injective map  $\mF_i'/G\to \tilde \mN_i/H_i'$. Suppose   $f, \tilde f \in \mF_i$ (resp. $\mF_i'$) with $\Lambda_f = \Lambda_{\tilde f}=H_i \cdot \lambda $ (resp. $H_i' \cdot \lambda$) then there exist $(r_1, \dots, r_4) \in \text{ord}(f)$ and $(\tilde r_1, \dots, \tilde r_4) \in \text{ord}(\tilde f)$  such that $\lambda = (r_1,r_2; r_3, r_4)=(\tilde r_1,\tilde r_2;\tilde r_3, \tilde r_4)$. By Lemma \ref{cross}, there exists a unique   $g \in PGL_2(\overline{\bF_q})$ which carries $(r_1,r_2,r_3,r_4)$ to $(\tilde r_1,\tilde r_2,\tilde r_3,\tilde r_4)$. However, from the definition of a restricted ordering, it follows that $\phi(g)$ also has  the same property, and hence $g \in PGL_2(q)$ by the uniqueness of $g$.  This shows that $\tilde f \in G \cdot f$. Thus, we have shown that the map $G \cdot f \mapsto \Lambda_f$ gives  injective maps   $\mF_i/G  \rightarrow \tilde \mN_i/H_i$ and  $\mF_i'/G \rightarrow \tilde \mN_i/H_2$. \\
It now remains to show that the maps $\mF_i/G \to \tilde \mN_i/H_i$ and $\mF_i'/G \to \tilde \mN_i/H_i'$ are surjective. We note that $\{-1\}$ forms an $H_i$-orbit in $\tilde \mN_i$, as well as an $H_i'$-orbit in $\tilde \mN_i$.  As mentioned in Lemma \ref{jmath_lem}, the stabilizer of each element in $\overline{\bF_q}$ under the $G_*$-action is trivial, and hence each $H_i$-orbit in $\mN_i$ has size $|H_i|$, and each $H_i'$-orbit in $\mN_i$ has size $H_i$. Thus $|\mN_i/H_i|=|\mN_i|/|H_i|$ equals $\tfrac{q-3}{6}, \tfrac{q-1}{2}$ and $\tfrac{q}{3}$ for $i=4,2$ and $1$, respectively. Similarly, 
$|\mN_i/H_i'|=|\mN_i|/|H_i'|$ equals $\tfrac{q-3}{2}$ and $\tfrac{q-1}{2}$ for $i=4$ and $2$ respectively. On the other hand, there is one orbit $G \cdot f$ in each $\mF \in \{\mF_4, \mF_2, \mF_1, \mF_4', \mF_2'\}$ given in Proposition \ref{j_zero_forms_prop}
with $\Lambda_f=\{-1\}$ (these forms have zero $\jmath$-invariant, and 
$\jmath(\lambda)= \frac{(\lambda+1)^6}{\lambda^2(\lambda-1)^2}=0$ if and only if $\lambda=-1$). We have also shown above that  the number of orbits of forms with non-zero $\jmath$-invariant in $\mF_4, \mF_2, \mF_1, \mF_4'$ and $\mF_2'$ are
$\tfrac{(q-3)}{6}$, $\tfrac{(q-1)}{2}$, $\tfrac{q}{3}$, $\tfrac{(q-3)}{2}$, and $\tfrac{(q-1)}{2}$, respectively. This completes the proof that the maps $\mF_i/G \to \tilde \mN_i/H_i$ and $\mF_i'/G \to \tilde \mN_i/H_i'$ are surjective.\end{proof}

For $r \in \bF_q^\times$, we define a quartic form $E_r$:
\beq \label{eq:E_r_def} E_r(X,Y)=X(Y^3-Y^2X+r^{-1} X^3), \quad r \in \bF_q^\times. \eeq
The $\jmath$-invariant of $E_r$ is $r$, and hence $G \cdot E_r$ represent $(q-1)$ distinct $G$-orbits in $\mF_4 \cup \mF_2  \cup \mF_1$. Since there are $\tfrac{(q-3)}{6}$, $\tfrac{(q-1)}{2}$, $\tfrac{q}{3}$ such orbits, we see that $\{G \cdot E_r \colon r \in \bF_q^\times\}$ is the complete set of $G$-orbits in $\mF_4 \cup \mF_2 \cup \mF_1$.

\begin{lem}\label{Ji_lem} The parts $J_i$  of the partition  $\bF_q^\times = J_4 \cup J_2 \cup J_1$ defined in Lemma \ref{Ji_def} can also be characterized as:
\begin{align*}
    J_2&=\{ r\in \bF_q^\times \colon r  \text{ is a non-square in $\bF_q$}\},\\
    J_4&=\{ r\in \bF_q^\times \colon r  \text{ is a square in $\bF_q$ and Tr$_{\bF_q/\bF_3}(1/\sqrt{r})=0$}\},\\
J_1&=\{ r\in \bF_q^\times \colon r  \text{ is a square in $\bF_q$ and Tr$_{\bF_q/\bF_3}(1/\sqrt{r})\neq 0$}\}.
\end{align*}
Further, the quartic form $E_r$ is in $\mF_i$ according as $r\in J_i$, where $i\in\{1,2,4\}$.
\end{lem}
\begin{proof}
We note that the $\jmath$-invariant of the above quartic form $E_r=X(Y^3-Y^2X+r^{-1} X^3)$ is $r$. Also $E_r$ is in $\mF_i$ for $ i \in \{1,2,4\}$, according as the cubic form $Y^3-Y^2X+r^{-1}X^3$, or equivalently the cubic polynomial $h(S)=S^3-rS+r$ (where $S=X/Y$) has $i-1$ roots in $\bF_q$.
   By \cite[Corollary 1.23]{Hirschfeld2}, the number of roots of $h(S)$ in $\bF_q$ is 
  \begin{enumerate}
     \item $1$ if $r$ is a non-square in $\bF_q$. There are $(q-1)/2$ such vales of $r$, 
     \item $3$  if $r$ is a square in $\bF_q$ and Tr$_{\bF_q/\bF_3}(1/\sqrt{r})=0$. There are $q/3 -1$ such values of $\sqrt{r}$ and hence $(q-3)/6$ such values of $r$,
     \item $0$  if $r$ is a square in $\bF_q$ and Tr$_{\bF_q/\bF_3}(1/\sqrt{r})\neq 0$. There are $2q/3$ such values of $\sqrt{r}$, and hence $q/3$ such values of $r$.
 \end{enumerate}
 We note that the sizes of the sets in (1), (2), (3) above agree with $|J_4|, |J_2|, |J_1|$. We will show that the sets in (1),(2) and (3) actually equal $J_4$, $J_2$ and $ J_1$, respectively. As noted above, $\{G \cdot E_r : r \in \bF_q^\times\}$ is a complete set of orbits in $\mF_4 \cup \mF_2 \cup \mF_1$  which have non-zero $\jmath$-invariant. For $i \in \{1,2,4\}$, the map $G \cdot f \mapsto \Lambda_f$ carries the set of all such orbits in $\mF_i$ bijectively to $\mN_i/H_i$ by Proposition \ref{Lambda_prop}. The function $\jmath(\lambda)$ carries the latter set  for $\mN_i/H_i$  bijectively to $J_i$ by Lemma \ref{Ji_def}. Thus, $E_r \in \mF_i$ if and only if $r \in J_i$.

\end{proof}
\subsection*{Proof of Theorem \ref{result2}} \hfill \\
Theorem \ref{result2} asserts that there are a total of $(2q+2)$  $G$-orbits of quartic forms with nonzero discriminant, whose sizes and $\jmath$-invariant are as given in Table \ref{table:1}, which we reproduce below.\\

The last row of this table follows from Proposition \ref{j_zero_forms_prop} which asserts that there are $5$ orbits of forms with zero $\jmath$-invariant, with one each in $\mF_4, \mF_4', \mF_2, \mF_2', \mF_1$ of sizes $|G|/24, |G|/8, |G|/4, |G|/4, |G|/3$, respectively. \\

We now turn to the first three rows of the table. 
Proposition \ref{Lambda_prop} 
asserts that the number of orbits of forms with non-zero $\jmath$-invariant in $\mF_i$ for $i \in \{1, 2, 4\}$ is in bijection with $\mN_i/H_i$, which in turn is in bijection with $J_i$ by Lemma \ref{Ji_def}. It also asserts that the number of orbits of forms with non-zero $\jmath$-invariant in $\mF_4'$ and $\mF_2'$ is in bijection with $\mN_4/H_2$ and $\mN_2/H_2$. By Lemma \ref{Ji_def}, we see that  $\mN_2/H_2$ is in bijection with $J_2$, whereas there is a $3$-to-$1$ map from $\mN_4/H_2 \to J_4$. Thus,  for $i \in \{1, 2, 4\}$,  there are a total of $i$ orbits for each $j \in J_i$ . By proposition \ref{stabilizer_prop}, the size of an orbit  $\mF_i, \mF_i'$ with non-zero $\jmath$-invariant is $|G|/i$.
\begin{table}[h!] 
\begin{tabular}{c| *{6}{c}}
   & $|G|$ & $\tfrac{|G|}{2}$ & $\tfrac{|G|}{3}$ & $\tfrac{|G|}{4}$ & $\tfrac{|G|}{8}$    &  $\tfrac{|G|}{24}$  \\
&&&&&&\\  \hline \\
$\jmath(f)\in J_4$ & $0$ & $0$ & $0$ & $4$ & $0$ &  $0$  \\
$\jmath(f)\in J_2$ & $0$ & $2$ & $0$ & $0$ & $0$ &  $0$  \\
$\jmath(f)\in J_1$ & $1$ & $0$ & $0$ & $0$ & $0$ &  $0$  \\
$\jmath(f)=0$  & $0$ & $0$ & $1$ & $2$ &  $1$  &  $1$
  \\ [1ex]
\hline  \\ [1ex]
\end{tabular}
 \end{table}

\subsection{Representatives  of orbits of forms with nonzero discriminant}
We end this section by tabulating quartic forms representing each of the $(2q+2)$ orbits of forms with nonzero discriminant. In Table \ref{table-BQF}, we list representatives for each of the $(2q +2)$ orbits of binary quartic forms with non-zero discriminant, and the isomorphism class of the stabilizer of each orbit (obtained in Proposition \ref{stabilizer_prop}).

By Proposition \ref{j_zero_forms_prop}, the $5$ orbits with $\jmath$ invariant zero in $\mF_4, \mF_2, \mF_1, \mF_4'$ and $\mF_2'$ are represented by  by $XY(Y^2-X^2)$, $XY(Y^2-\ep X^2)$,
$X(Y^3-YX^2 -r_0 X^3)$ (where Tr$_{\bF_q/\bF_3}(r_0) \neq 0$),  $(Y^2-\epsilon X^2)(Y-\alpha X)(Y - \alpha^q X)$ (where  $\alpha \alpha^q=\epsilon$), and 
\[f(X,Y)=\begin{cases} 
(Y^4-\gamma X^4) &\text{ if $q \equiv 1 \mod 4$},\\
(Y^4-X^4 +X^3Y) &\text{ if $q \equiv 3 \mod 4$},\end{cases}
\]
respectively.

As observed in the discussion preceding, Lemma \ref{Lambda_prop},  
for $i \in \{1, 2, 4\}$, the complete set of $G$-orbits in $\mF_i$ with non-zero $\jmath$-invariant is $\{G \cdot E_r \colon r \in J_i\}$. Here 
\[  E_r(X,Y)=X(Y^3-Y^2X+r^{-1} X^3), \quad r \in \bF_q^\times. \]

We now turn to the representatives of the orbits in $\mF_4'$. We recall from the proof of Proposition \ref{Lambda_prop}, that if $(r_1, r_2, r_3,r_4)=(r_1, \phi(r_1), r_3 ,\phi(r_3))$ is an ordering of the roots of $f \in \mF_4'$, then $\lambda_f=N_{\bF_{q^2}/\bF_q}(\beta)$ where $\beta = \tfrac{r_3-r_1}{r_3-\phi(r_1)}$. For each $\{\lambda, \lambda^{-1}\} \in \mN_4/H_2$, we choose some  $\alpha=\alpha_{ \{\lambda, \lambda^{-1}\}} \in \bF_{q^2}$ such that 
\[ N_{\bF_{q^2}/\bF_q}(\tfrac{\alpha - \sqrt\ep}{\alpha+\sqrt\ep})=\lambda.\]
Then 
\[ \psi_\alpha = (Y^2- \ep X^2) ( Y - \alpha X)(Y - \phi(\alpha)X)  \in \mF_4',\]
 has roots $(r_1,\phi(r_1),r_3,\phi(r_3))=(\sqrt\ep,-\sqrt\ep, \alpha, \phi(\alpha))$. The quantity $\beta = \tfrac{r_3-r_1}{r_3-\phi(r_1)}$ equals $\tfrac{\alpha-\sqrt\ep}{\alpha+\sqrt\ep}$ and hence $\Lambda_{\psi_\alpha}=\tfrac{(\alpha-\sqrt\ep)(\alpha^q+\sqrt\ep)}{(\alpha+\sqrt\ep)(\alpha^q-\sqrt\ep)}$ is represented by $N_{\bF_{q^2}/\bF_q}(\beta)= \lambda$ as required.   \\

We now turn to the representatives of the orbits in $\mF_2'$. We recall from the proof of Proposition \ref{Lambda_prop}, that
if 
 $(r_1, r_2, r_3,r_4)=(r_1, \phi^2(r_1), \phi(r_1), \phi^3(r_1))$ is an ordering of the roots of $f \in \mF_2'$,   then $\lambda_f=\tfrac{\beta}{\phi(\beta)}$ where $\beta=N_{\bF_{q^4}/\bF_{q^2}}(\phi(r_1)-r_1)$. 

 If $q \equiv 1 \mod 4$, let $\bF_{q^4} = \bF_q[\theta]$ where $\theta^4=\gamma$ and $\gamma$ is a generator of the cyclic group $\bF_q^{\times}$. We note that $\phi(\theta) = \imath \theta$ where $\imath \in \bF_q$ is a square root of $-1$. The map  $g(t) = \tfrac{t +  \imath \theta^2}{t - \imath \theta^2}$  maps $\bF_q^\times$ bijectively to $\mN_2$. Therefore, it suffices to exhibit $f \in \mF_2'$ for which $\lambda_f$ is of the form $\tfrac{r +  \imath \theta^2}{r - \imath \theta^2}$.
 We claim $f = \upsilon_r$ defined below, has this property: 
\beq \label{eq:upsilonr1mod4}   \upsilon_r= \begin{cases}
 (Y^2- r  X^2)^2 - \gamma X^4 -   \sqrt{\gamma r}  X^3Y &\text{if $\ep r \in (\bF_q^\times)^2$,} \\
  (Y^2-\gamma r X^2)^2 - \gamma^3 X^4 -   \gamma^2 \sqrt{r}   X^3Y &\text{if $r \in (\bF_q^\times)^2$}.
\end{cases} \eeq

To see this, we consider $(r_1, \phi^2(r_1),\phi(r_1), \phi^3(r_1))$ with

\begin{enumerate}
\item $r_1=-\imath \theta-  t \theta^2$ where $t \in \bF_q^\times$.
Here  $\phi(r_1) - r_1=(\imath+1)\theta - t \theta^2$ and hence $\beta=N_{\bF_{q^4}/\bF_{q^2}}(\phi(r_1)-r_1)$ equals $\imath \theta^2+t^2 \gamma $. Therefore, $\lambda_f=\tfrac{\beta}{\phi(\beta)}$ equals $\tfrac{\gamma t^2 + \imath \theta^2}{\gamma t^2 - \imath \theta^2}$. 
The corresponding form in $\mF_2'$ is:
\[  \upsilon_{\ep t^2}=(Y^2-\gamma t^2 X^2)^2 - \gamma X^4 -   \gamma t  X^3Y.\]

\item   $r_1=t \theta^2+\theta^3$  where $t \in \bF_q^{\times}$.
Here,  we have $\phi(r_1) - r_1=-\theta^2( -t+ (\imath+1)\theta )$ and hence $\beta=N_{\bF_{q^4}/\bF_{q^2}}(\phi(r_1)-r_1)$ equals $ \gamma (t^2 + \imath \theta^2)$. Therefore, $\lambda_f=\tfrac{\beta}{\phi(\beta)}$ equals $\tfrac{ t^2 + \imath \theta^2}{ t^2 - \imath \theta^2}$. The corresponding  forms in $\mF_2'$ is:
\[ \upsilon_{t^2}= (Y^2-\gamma t^2 X^2)^2 - \gamma^3 X^4 -  \gamma^2 t  X^3Y.\]
\end{enumerate}

The map
 \[g: \bF_q^\times   \to \mN_2 , \qquad g(r)=\tfrac{r +  \imath \theta^2}{r - \imath \theta^2} \]
is bijective, and moreover maps the classes  $\{\pm r\}$ bijectively to the classes $\{\lambda, \lambda^{-1}\}$. We pick some  $r_{\{\lambda, \lambda^{-1}\}} \in \bF_q^\times$ satisfying the property that 
$g(r_{\{\lambda, \lambda^{-1}\}}) \in \{\lambda, \lambda^{-1}\}$. For example,  
\[ r_{\{\lambda, \lambda^{-1}\}} =\tfrac{\imath \theta^2 (\lambda+1)}{(\lambda-1)}. \]
We will consider the forms  
$\upsilon_r$   for these $(q-1)/2$ values of $r$, where  $\upsilon_r$ 
is the form given in \eqref{eq:upsilonr1mod4}.

If $q \equiv 3$ mod $4$,  let $\bF_{q^2}=\bF_q[\imath]$ where $\imath^2=-1$ and let $\bF_{q^4} = \bF_{q^2}[\theta]$ where $\theta^2=\imath-1$. Since  $N_{\bF_{q^2}/\bF_q}(\imath-1)=-1$ is a non-square in $\bF_q$, it follows that $\imath-1$ is a non-square in $\bF_{q^2}$. We note that $\phi(\theta) = \imath/\theta$. 
 The map  $g(r) =  \tfrac{  1+r  -  \imath}{  1 +r +  \imath }$  maps $\bF_q\setminus\{-1\}$ bijectively to $\mN_2$. Therefore, it suffices to exhibit $f \in \mF_2'$ for which $\lambda_f$ is of the form
$\tfrac{  1 +r-  \imath}{   1 +r + \imath }$ for each $r \in \bF_q\setminus\{-1\}$.
We claim $f = \upsilon_r$ defined below, has this property: 
\beq \label{eq:upsilonr3mod4}  \upsilon_r=\begin{cases}
 (Y^2- r  X^2)^2 -  X^4  +  X^2(  (Y^2+ r  X^2)     -XY \sqrt{-r})   &\!\!\text{if $r \in \{-x^2: x \in \bF_q\setminus\{\pm 1\}\}$}, \\
  (Y^2+ r X^2)^2 -X^4 -   X^2(  (Y^2- r  X^2)   + XY \sqrt{r})  &\!\!\!\!\text{if $r \in (\bF_q^\times)^2$.}
\end{cases} \eeq
 
To see this, we consider $(r_1, \phi^2(r_1),\phi(r_1), \phi^3(r_1))$ with
\begin{enumerate}
\item $r_1=\imath(\theta+  t)$ where $t \in \bF_q\setminus\{\pm 1\}$.  \\Here  $\phi(r_1) - r_1= ( \imath t - \imath \theta +1/\theta)$ and hence $\beta=N_{\bF_{q^4}/\bF_{q^2}}(\phi(r_1)-r_1)$ equals $- t^2+ 1  -  \imath$.  Therefore, $\lambda_f=\tfrac{\beta}{\phi(\beta)}$ equals
$\tfrac{ -t^2 +1  -  \imath}{ -t^2 +1  +  \imath }$. The corresponding form in $\mF_2'$ is:
\[  \upsilon_{- t^2}= (Y^2+ t^2  X^2)^2 -  X^4  +X^2(  (Y^2- t^2  X^2)   - tXY  ). \]

\item $r_1=(-\imath t -\imath/ \theta)$ for $ t \in  \bF_q^{\times}$. Here  $\phi(r_1) - r_1= (- \imath t + \imath/ \theta + \theta)$ and hence $\beta=N_{\bF_{q^4}/\bF_{q^2}}(\phi(r_1)-r_1)$ equals $- t^2-   1 +  \imath$.  Therefore, $\lambda_f=\tfrac{\beta}{\phi(\beta)}$ equals
$\tfrac{ t^2 +1 -  \imath}{  t^2 +1 +  \imath }$. The corresponding form in $\mF_2'$ is:
\[  \upsilon_{t^2}= (Y^2+ t^2  X^2)^2 -  X^4  - 2 X^2(  (Y^2- t^2  X^2) \theta_0 - 2t XY  \theta_1).   \]

\end{enumerate}

The map 
\[ g: \bF_q\setminus \{-1\}   \to \mN_2 , \qquad g(r)=\tfrac{  1+r  -  \imath}{  1 +r +  \imath }\]
is bijective, and moreover maps the classes  $\{r, 1-r\}$ bijectively to the classes $\{\lambda, \lambda^{-1}\}$. 
We again pick  $r_{\{\lambda, \lambda^{-1}\}} \in \bF_q$ which satisfies the property that 
$g(r_{\{\lambda, \lambda^{-1}\}}) \in \{\lambda, \lambda^{-1}\}$, for example,
\[ r_{\{\lambda, \lambda^{-1}\}} =-1+\tfrac{\imath  (\lambda+1)}{(1-\lambda)}. \]  
We will consider the forms  
$\upsilon_r$   for these $(q-1)/2$ values of $r$, where  $\upsilon_r$ 
is the form given in \eqref{eq:upsilonr3mod4}.

\begin{table}[h]
\begin{tabular}{||c|c|c|c||}
\hline
Representatives of $\fO$ & $\jmath(\fO)$ & Type & Stabilizer \\
&&&\\  
\hline  
&&&\\  
$E_r, \;  r \in J_4$ \; ($|J_4|$ orbits)& $r \in J_4$ & $\mF_4$&  $\bZ/2\bZ \times \bZ/2\bZ$ \\
&&&\\  
$E_r, \;  r \in J_2$  \; ($|J_2|$ orbits)& $r\in J_2$ & $\mF_2$&  $\bZ/2\bZ$ \\
&&&\\  
$E_r, \;  r \in J_1$  \; ($|J_1|$ orbits)& $r\in J_1$ &$ \mF_1$&  trivial \\
&&&\\
$\psi_{\alpha_{\{\lambda, \lambda^{-1}\}}}$, $\{\lambda, \lambda^{-1}\}\in \mN_4/H_2$  ($3|J_4|$ orbits) & $\jmath(\lambda) \in J_4$ & $\mF_4'$&  $\bZ/2\bZ \times \bZ/2\bZ$ \\
&&&\\
$\upsilon_{r_{\{\lambda, \lambda^{-1}\}}}, \quad \{\lambda, \lambda^{-1}\} \in \mN_2/H_2$
& $ \jmath(\lambda) \in J_2$& $\mF_2'$ &$\bZ/2\bZ$\\
$q \equiv 1 \mod 4$ these are $|J_2|$ orbits 
&&&\\ 
&&&\\
$\upsilon_{r_{\{\lambda, \lambda^{-1}\}}}, \quad \{\lambda, \lambda^{-1}\} \in \mN_2/H_2$  & $ \jmath(\lambda)\in J_2$& $\mF_2'$ &$\bZ/2\bZ$\\
$q \equiv 3 \mod 4$ these are $|J_2|$ orbits
&&&\\
\hline
&&&\\
$XY(Y^2-X^2)$ & $0$ & $\mF_4$ & $S_4$ \\
&&&\\
$XY(Y^2-\ep X^2)$ & $0$ & $\mF_2$ & $\bZ/2\bZ \times \bZ/2\bZ$\\
&&&\\
$X(Y^3-YX^2-r_0X^3)$, $Tr_{\bF_q/\bF_3}(r_0)\neq 0$ & $0$ & $\mF_1$ & $\bZ/3\bZ$\\
&&&\\
$(Y^2-\ep X^2)(Y-\alpha X)(Y-\alpha^q X)$, $\alpha \alpha^q=\ep$ & $0$ & $\mF_4'$ & $D_4$\\
&&&\\
$\begin{cases} 
(Y^4-\gamma X^4) &\text{ if $q \equiv 1 \mod 4$}\\
(Y^4-X^4 +X^3Y) &\text{ if $q \equiv 3 \mod 4$}\end{cases}$ & $0$ & $\mF_2'$  & $\bZ/4\bZ$ \\
\hline 
\end{tabular}
\[\]
\caption{Representatives and stabilizers of $G$-orbits of binary quartic forms with discriminant nonzero}
 \label{table-BQF}
 \end{table}

\section{$G$-orbits of lines of $PG(3,q)$} \label{four}
For notational simplicity, we denote the projective space $\bP(D_3V)$ as just $PG(3,q)$. The lines of $PG(3,q)$ decompose into $8$ classes, denoted $\mO_1, \dots, \mO_8$ in   Hirschfeld's book  \cite[Lemma 21.1.4]{Hirschfeld3}. The class $\mO_6$ consists of the generic lines. The classes of non-generic lines have been completely decomposed into $G$-orbits in literature. We follow the notation of Blokhuis-Pellikaan-Sz\H{o}nyi, in \cite[Theorem 8.1]{BPS}: the classes except $\mO_5$ and $\mO_8$ form a single orbit where as  $\mO_5$ and $\mO_8$ decompose into orbits $\mO_5^+ \cup \mO_5^-$
and $\mO_{8.1}^+\cup \mO_{8.2}^- \cup \mO_{8.2}$ respectively. In all the non-generic lines decompose into ten  $G$-orbits.
\subsection{$G$-orbits of non-generic lines of $PG(3,q)$}
We recall that a line is non-generic if it meets $C \cup \mA$. A line $L$ intersects $\mA$ if and only if $z_2(f_L)=0$.

\subsubsection{Lines intersecting $\mA$}
We have seen in Proposition \ref{correspondence} that   $L \neq \mA$ intersects $\mA$, if and only if $f_L$ has a linear factor over $\bF_q$ of multiplicity at least $3$. Thus $f_L$ is in the $G$-orbit of $X^4$ or $X^3Y$.  Since the map $\bP((\wedge^2 D_3V)/\hat A) \to \text{Sym}^4(V^*)$ is $G$-equivariant, we see that $L$ is  the $G$-orbit of either i) $(z_0, \dots, z_5)=(1,0,0,0,0,z_5)$ for some $z_5 \in \bF_q$, or ii) $(z_0, \dots, z_5)=(0,1,0,0,0,z_5)$ for some $z_5 \in \bF_q$. It is easy to see from the matrices \eqref{eq:g_4}, that these form $5$ orbits: 
 \begin{enumerate}
\item  $\mO_2=G \cdot (1,0,0,0,0,0)$ which consists of the $(q+1)$ tangents of $C$. 
\item  $\mO_4=G \cdot (0,1,0,0,0,0)$ which consists of the $q(q+1)$ non-tangent unisecants contained in some osculating planes of $C$.
\item The remaining $3$ orbits form the  class $\mO_8$ of external lines contained in some osculating plane of $C$:
\begin{enumerate}
\item  $\mO_{8.1}^+=G \cdot (1,0,0,0,0,1)$ of size $(q^2-1)/2$. 
\item   $\mO_{8.1}^-=G \cdot (1,0,0,0,0,\ep)$ of size $(q^2-1)/2$.
\item  $\mO_{8.2}=G \cdot (0,1,0,0,0,1)$ of size $(q^3-q)$. 
\end{enumerate}
\end{enumerate}

\subsubsection{Lines meeting $C$ but not $\mA$}
There are $q^3+q^2-q$ such lines. As observed in Proposition \ref{correspondence},  the map $L \leftrightarrow f_L$ expressed in coordinates as 
\[ (z_0, \dots, z_4, (z_1z_3-z_0z_4)/z_2) \mapsto z_0X^4+z_1X^3Y+z_2X^2Y^2+z_3XY^3+z_4Y^4,\] induces a  bijective correspondence between $G$-orbits of lines meeting $C$ but not $\mA$, and the $G$-orbits of quartic forms having a linear form of multiplicity exactly $2$.   From Lemma \ref{disc_zero_forms},
there are $4$ such orbits $G \cdot (X^2Y^2)$, $G \cdot X^2Y(Y-X)$, $G \cdot X^2(Y^2-\ep X^2)$ and $G \cdot (Y^2-\ep X^2)^2$ of sizes $q(q+1)/2$, $|G|/2$,  $|G|/2$, and $q(q-1)/2$ respectively. Therefore we get the following $4$ orbits of such  lines:
\begin{enumerate}
\item $\mO_1=G\cdot (0,0,1,0,0,0)$ of size $q(q+1)/2$ consisting of the real chords of $C$, and corresponding to $G \cdot (X^2Y^2)$.

\item  $\mO_3=G \cdot (\ep^2,0,\ep,0,1,-\ep)$ of size $q(q-1)/2$ consisting of the imaginary chords of $C$, and corresponding to $G \cdot (Y^2-\ep X^2)^2$.

\item The remaining $2$ orbits form the class $\mO_5$ of non-tangent unisecants not contained in any osculating plane:
\begin{enumerate}
\item  $\mO_5^+=G \cdot (0,0,-1,1,0,0)$ of size $\tfrac{|G|}{2}$  corresponding to $G \cdot Y^2X(Y-X)$.
\item $\mO_5^-=G \cdot (0,0,-\ep,0,1,0)$ of size  $\tfrac{|G|}{2}$ corresponding to $G \cdot Y^2(Y^2-\ep X^2)$.
\end{enumerate}
\end{enumerate}


\subsection{$G$-orbits of generic lines} 
\subsubsection{Proof of Theorem \ref{thm_main}}
\bep By Proposition \ref{correspondence}, the map $L \mapsto f_L$ gives a bijective correspondence between $G$-orbits of generic lines of $PG(3,q)$ and $G$-orbits of quartic forms with  nonzero $\jmath$-invariant.  Thus, from Table \ref{table:1}, for each $i \in \{1, 2, 4\}$, we have $i |J_i|$ orbits of size $|G|/i$. \eep


\subsubsection{Generators of the line orbits}
Recall from Section \ref{two} that, if $L$ be a generic line and $f_L$ be the binary quartic form associated with it, then the Pl\"ucker coordinates of $L$ can be given by
\beq \label{change-coordinate}
(p_{01},p_{02},p_{03},p_{12},p_{13},p_{23})=(z_4,z_3,z_2,(z_1z_3-z_0z_4)/z_2, z_1, z_0).
\eeq
Let $\mO$ be the orbit of a generic line $L$ corresponding to the orbit $\fO$ of the associated binary quartic form $f_L$. Let $\jmath(\mO)=\jmath(\fO)$, where $\jmath(\fO)=\jmath(f_L)$.

In Table \ref{table-Lines} below, we list the representatives for all the $(2q-3)$ $G$-orbits of generic lines in $PG(3,q)$. The first column lists the representatives for the orbit $\mO$ of generic lines, the second column gives the value $\jmath(\mO)$, the third column specifies the corresponding type of binary quartic forms, and the fourth column indicates the isomorphism class of the stabilizer of the orbit $\mO$. The usual way to represent lines of $PG(3,q)$ is to give a pair of generators in $PG(3,q)$ for the line. We here represent these generators by the rows of a $2\times 4$ matrix, and called it as a generator matrix for that line.

We begin with the orbits of generic lines for which the associated orbits of binary quartic forms is in $\mF_4 \cup \mF_2 \cup \mF_1$. Let $E_r(X,Y)=X(Y^3-Y^2X+r^{-1} X^3)$, where $\jmath(E_r)=r\in \bF_q^\times$. By Lemma \ref{Ji_lem}, we know that $E_r\in \mF_i$ according as $r\in J_i$, where $i\in \{1,2,4\}$. By (\ref{change-coordinate}), the Pl\"ucker coordinates of the line $L$ associated with $E_r$ is $(0,1,-1,0,0,r^{-1})$. Therefore, a generator matrix for $L$ can be given by
\[M(E_r)=\bbm 1 & 0 & 0 & -r^{-1} \\ 0 & 0 & 1 & -1 \bem.\]

We recall from Table \ref{table-BQF}, the representative $\psi_{\alpha_{\{\lambda, \lambda^{-1}\}}}$ for $\lambda \in \mN_4$, for the orbits in $\mF_4'$. Let $\alpha_{\{\lambda, \lambda^{-1}\}}=x+y\sqrt{\ep}$, $x,y\in \bF_q$, then by (\ref{change-coordinate}), the Pl\"ucker coordinates of the line associated with $\psi_{\alpha_{\{\lambda, \lambda^{-1}\}}}$ is $(1,x,(x^2-\ep y^2)-\ep,\tfrac{-\ep^2y^2}{(x^2-\ep y^2)-\ep},-\ep x,\ep(\ep y^2-x^2))$. A generator matrix for the corresponding orbits of lines  are:
\[M(\psi_{\alpha_{\{\lambda, \lambda^{-1}\}}})=\bbm 1 & 0 & \tfrac{\ep^2y^2}{(x^2-\ep y^2)-\ep} & \ep x \\ 0 & 1 & x & (x^2-\ep y^2)-\ep \bem.\]

Again from Table \ref{table-BQF}, we have the representatives $\upsilon_{r_{\{\lambda, \lambda^{-1}\}}}$ for  $\lambda \in \mN_2$, for the orbits in $\mF_2'$.  A Generator matrix for the corresponding each of the orbits of lines are as below.  The cases $q \equiv 1 \mod 4$ and $q \equiv 3 \mod 4$ are treated separately.
The notation $M(\upsilon_{r_{\{\lambda, \lambda^{-1}\}}}, 1)$ and  $M(\upsilon_{r_{\{\lambda, \lambda^{-1}\}}}, 3)$ refer to these two cases.
\begin{equation*}
  M(\upsilon_{r_{\{\lambda, \lambda^{-1}\}}}, 1)  = \begin{cases} 
  \bbsm 1& 0& \tfrac{r^2-\gamma}{r} & \sqrt{\gamma r}\\
 0&1& 0&  r \besm  &\text{if $\ep r\in (\bF_q^\times)^2$}, \\
  \bbsm 1& 0& \tfrac{\gamma(r^2-\gamma)}{r} &  \gamma^2 \sqrt{r}\\
 0&1& 0&   \gamma r \besm &\text{if $ r\in (\bF_q^\times)^2$},\end{cases}
 \end{equation*}

\begin{equation*}
  M(\upsilon_{r_{\{\lambda, \lambda^{-1}\}}}, 3)  = \begin{cases} 
  \bbsm 1& 0& \tfrac{r^2+r-1}{r+1} & \sqrt{- r}\\
 0&1& 0&  r+1 \besm  &\text{if $ r \in \{-x^2: x \in \bF_q\setminus\{\pm 1\}\}$}, \\
  \bbsm 1& 0& -\tfrac{r^2+r-1}{r+1} & \sqrt{r}\\
 0&1& 0&   -(r+1) \besm &\text{if $ r\in (\bF_q^\times)^2$}. \end{cases} 
 \end{equation*}

\begin{table}[h]
\begin{tabular}{||c|c|c|c||}
\hline
Generators of $\mO$ & $\jmath(\mO)$ & Type & Stabilizer \\
&&&\\  
\hline  
&&&\\  
$M(E_r), \;  r \in J_4$ \; ($|J_4|$ orbits)& $r \in J_4$ & $\mF_4$&  $\bZ/2\bZ \times \bZ/2\bZ$ \\
&&&\\  
$M(E_r), \;  r \in J_2$  \; ($|J_2|$ orbits)& $r\in J_2$ & $\mF_2$&  $\bZ/2\bZ$ \\
&&&\\  
$M(E_r), \;  r \in J_1$  \; ($|J_1|$ orbits)& $r\in J_1$ &$ \mF_1$&  trivial \\
&&&\\
$M(\psi_{\alpha_{\{\lambda, \lambda^{-1}\}}}), \;  \lambda \in \mN_4$ \; ($3|J_4|$ orbits) & $\jmath(\lambda) \in J_4$ & $\mF_4'$&  $\bZ/2\bZ \times \bZ/2\bZ$ \\
&&&\\
$M(\upsilon_{r_{\{\lambda, \lambda^{-1}\}}}, 1) \;  \lambda \in \mN_2$ \; ($|J_2|$ orbits)& $\jmath(\lambda) \in J_2$ & $\mF_2'$&  $\bZ/2\bZ$ \\
Here $q \equiv 1 \mod 4$
&&&\\
&&&\\
$M(\upsilon_{r_{\{\lambda, \lambda^{-1}\}}}, 3) \;  \lambda \in \mN_2$ \; ($|J_2|$ orbits)& $\jmath(\lambda) \in J_2$ & $\mF_2'$&  $\bZ/2\bZ$ \\
Here $q \equiv 3 \mod 4$
&&&\\  
\hline 
\end{tabular}
\[\]
\caption{Generators  and stabilizers of $G$-orbits of generic lines of $PG(3,q)$}
 \label{table-Lines}
 \end{table}

\section{Elliptic curve associated to quartic forms  with non-zero $\jmath$-invariant} \label{five}
Let $f \in \bP(\text{Sym}^4(V^*))$ be a quartic form with $\jmath(f) \neq 0$. We  recall that $\jmath(f) \neq 0$ if and only if the coefficient $z_2(f)$ of $X^2Y^2$ in $f$ is nonzero.  We consider the quartic form (which we again denote as $f$) in Sym$^4(V^*)$  representing $f$, and   having $1$ as the coefficient of $X^2Y^2$.
Let 
\beq \label{eq:zetaf} \zeta_f=\#\{(x,y) \in PG(1,q) \colon f(x,y) \text{ is a non-zero square in $\bF_q$} \}.\eeq
In the next theorem, we  express $\zeta_f$ in terms of the number $\#\mE_r(\bF_q)$ of $\bF_q$-rational points of the elliptic curve in $\bP^2$ defined by:
\beq \label{eq:mEr} \mE_r: S^2=T^3+T^2 -r^{-1}, \quad  r=\jmath(f).\eeq
The $\jmath$-invariant of the elliptic curve $\mE_r$ is also $r$, as mentioned in Section \S1. If $f$ has at least one linear factor over $\bF_q$, that is, $f \in \mF_1 \cup \mF_2 \cup \mF_4$, then the result below is straightforward. However, for $f \in \mF_2' \cup \mF_4'$ the relation between $\zeta_f$ and $\#\mE_r(\bF_q)$ is not so obvious. 
\begin{thm} \label{thm_elliptic} Let $f$ be a quartic form with $\jmath(f)=r \neq 0$. The quantity $\zeta_f$ defined above is related to the number $\#\mE_r(\bF_q)$ of $\bF_q$-rational  points of the elliptic curve $\mE_r$ defined above by
   \[ \zeta_f = \begin{cases}
       (\#\mE_r(\bF_q)-i)/2 & \text{ if $f\in \mF_i$ for $i \in \{1, 2, 4\}$} \\
       \#\mE_r(\bF_q)/2 & \text{ if $f\in \mF_2' \cup \mF_4'$.}
   \end{cases}\] 
\end{thm}

\bep
For $g =  \bbsm a & b \\ c &d \besm \in GL_2(\bF_q)$, we note that $ g \cdot f = \det(g)^{-4} f(dx-by,ay-cx)$, and the coefficient $z_2(g \cdot f)$ of $X^2Y^2$ in $g \cdot f$ is 
$\det(g)^{-2} z_2(f)$. Therefore, 
\[\zeta_{g \cdot f}=\#\{(x,y) \in PG(1,q) \colon (ad-bc)^{-2} f(dx-by,ay-cx) \text{ is a non-zero square in $\bF_q$} \},\]
equals $\zeta_f$. This shows that $\zeta_f$ only depends on the orbit $G \cdot f$. In case $f \in \mF_1 \cup \mF_2 \cup \mF_4$, we may take $f= X(-Y^3+Y^2X-r^{-1} X^3))$. Let $T=-Y/X$. If $r \in J_i$, then $T^3+T^2-r^{-1}$ has $(i-1)$ roots in $\bF_q$. Therefore, $\# \mE_r(\bF_q)$ equals  $1+ (i-1) + 2 \zeta_f$, where the contribution $1$ comes from the point at infinity of $\mE_r(\bF_q)$. Therefore, $\zeta_f = (\#\mE_r(\bF_q)-i)/2$.\\

We turn to the case when  $f \in \mF_2' \cup \mF_4'$, i.e. the case when $f$
has no linear factor over $\bF_q$.  In particular,  $f(x,y) \neq 0$ for all $(x,y) \neq (0,0) \in \bF_q \times \bF_q$. We consider  the  curve in $\mathbb  P^2$
\[ X^2W^2= f(X,Y).\]
Since  $f$ has no repeated factors, the only singularity of this curve is the point $(X,Y,W)=(0,0,1)$, which is a cusp singularity. We define $\mE_f$ to be a non-singular model of this  curve. The curve $\mE_f$ has genus $1$, and the singular point of the original curve corresponds to a pair of points of $\mE_f$ around which a  model of $\mE_f$ is:
\[  (\frac{WX}{Y^2})^2=z_4 + z_3 (\frac{X}{Y}) + (\frac{X}{Y})^2+z_1(\frac{X}{Y})^2 +z_0 (\frac{X}{Y})^4. \]
The above mentioned pair of points is $(\tfrac{WX}{Y^2},\tfrac{X}{Y}) = (\pm \sqrt{z_4},0)$.
These $2$ points are defined over $\bF_q$ if and only if  $z_4=f(0,1)$ is a  square  in $\bF_q$. We also note the remaining points of $\mE_f(\bF_q)$ are 
\[ \{ (X,Y,W)=(1,y, \pm \sqrt{f(1,y)})   \colon f(1,y) \text{ is a square}\}.\]
Therefore,
\[ \zeta_f=\# \mE_f(\bF_q)/2.\]
By the Hasse bound, $\# \mE_f(\bF_q) \geq (\sqrt q -1)^2 >0$ and hence $\zeta_f >0$. By the definition of $\zeta_f$, we conclude that there is a point $(x_0,y_0) \in PG(1,q)$ such that $f(x_0,y_0)$ is a non-zero square. Let $g \in G$ such that $g (x_0,y_0)=(0,1)$. Since $\zeta_f=\zeta_{g \cdot f}$, we may assume $(x_0,y_0)=(0,1)$ or equivalently $f(0,1)=z_4$ is a non-zero square. We  write $f(X,Y)=z_0 X^4+ z_1X^3Y+X^2Y^2 + z_3 XY^3 + z_4 Y^4$, with $z_4$ a square. Further, replacing $f$ by $g \cdot f$, where $g = \bbsm \sqrt{z_4}^{-1} & 0\\1 & 1 \besm$, we may assume $f(X,Y)=z_0 X^4+ z_1X^3Y+ z_3 XY^3 +  Y^4+X^2Y^2$. Again replacing $f$ by $g \cdot f$ where $g=\bbsm 1 & 0\\z_3 & 1\besm$, we may assume
$f(X,Y)=z_0 X^4+ z_1X^3Y +  Y^4+X^2Y^2$. We note that the $\jmath$-invariant $r$ of $f$ only depends on the $G$-orbit of $f$ and hence
\[ r= \frac{z_2(f)^6}{\Delta(f)}= (z_0(z_0-1)^2 - z_1^2)^{-1}.\]
We take  $\mE_f$ to be a non-singular model of the projective plane curve
\[X^2W^2= z_0 X^4+ z_1X^3Y+ Y^4+X^2Y^2.\]
We show that the curves $\mE_f$ and $\mE_r$ are isomorphic. Our  main idea is inspired by Theorem 2 in \S 10 of  Mordell's book \cite{Mordell}.
Let $P_+$ and $P_-$ denote the points of $\mE_f$ with coordinates $(XW/Y^2, X/Y)$ being $ (1,0)$ and $(-1,0)$ respectively. Let $P_\infty$ denote the point at infinity of $\mE_r$ with coordinates $(1/S,T/S)=(0,0)$.  We define a map 
$\psi: \mE_f \to \mE_r$ as follows: away from $\{P_+, P_-\}$, we define $\psi$ in terms of the coordinates $(S,T)$ on $\mE_r$ by:
\[ \psi: (1,Y,W) \mapsto (S,T)=((Y^2+W-1)Y+z_1,(z_0-1)-(Y^2+W-1)).\]
We  check that $\psi(1,Y,W)$ lies on $\mE_r$: 
we can write the equation of $\mE_f$ as 
\[\mE_f: (Y^2-1)^2 -W^2 + z_0+z_1Y-1=0,\]
and we can write the equation of $\mE_r$ in terms of  $(\tilde S, \tilde T)=(S-z_1,T-(z_0-1))$  as
\[ \mE_r: \tilde T^3+\tilde T^2 -\tilde S^2 +z_1 \tilde S +(1-z_0) \tilde T=0, \]
(where we have used the fact that $r^{-1}=-z_1^2+(z_0-1)^2z_0$). Therefore, for $(\tilde S, \tilde T)=((Y^2+W-1)Y,-(Y^2+W-1))$, we get 
$\tilde T^3+\tilde T^2 -\tilde S^2 +z_1 \tilde S +(1-z_0) \tilde T$  evaluates to 
\[ (Y^2+W-1)( (Y^2-1)^2 -W^2 +z_1 Y+z_0-1)=0. \]

Near $P_-$, writing the equation of $\mE_f$ as
\[W+Y^2-1=\frac{z_0/Y^2+z_1/Y-1-W/Y^2}{W/Y^2 -1},\]
we see that $W+Y^2-1$ evaluates to $0$ at $P_-$. Similarly, writing
\[ Y(W+Y^2-1)=\frac{(z_0-1)/Y + z_1 - (Y^2+W-1)/Y}{W/Y^2-1},\]
we see that $Y(W+Y^2-1)$ evaluates to $z_1$ at $P_-$. Therefore $\psi(P_-)=(-z_1,z_0-1)$.\\

Near $P_+$, we can express $\psi$ in terms of the coordinates $(1/S,T/S)$ near $P_{\infty}$ as:
\[ (1/S,T/S)=(\frac{\frac{1}{Y^3}}{(1+\frac{W}{Y^2}-\frac{1}{Y^2})+\frac{z_1}{Y^3}}, \frac{\frac{z_0-1}{Y^3}-(1+\frac{W}{Y^2}-\frac{1}{Y^2})/Y}{(1+\frac{W}{Y^2}-\frac{1}{Y^2})+\frac{z_1}{Y^3}}).\]
Since $(1+\frac{W}{Y^2}-\frac{1}{Y^2})$ and $1/Y$ evaluate to $-1$ and $0$ respectively at $P_+$, we see that the right hand side above evaluates to $(1/S,T/S)=(0,
0)$ at $P_+$, or in other words $\psi(P_+)=P_\infty$. We also note that $\psi^{-1}(z_1,z_0-1)$ is the unique point $(1,Y,W)=(1,\tfrac{1-z_0}{z_1}, 1-(\tfrac{1-z_0}{z_1})^2)$ if $z_1 \neq 0$, and $P_-$ if $z_1=0$.\\

We define a map $\psi':\mE_r \to \mE_f$  by 
\[\psi'\colon  (S,T) \mapsto (1,Y,W)=(1, \tfrac{S-z_1}{z_0-1-T}, z_0-T-(\tfrac{S-z_1}{z_0-1-T})^2 ),  \]
away from the points $(S,T)=\{ (\pm z_1,z_0-1), P_\infty \}$. It is readily checked that the point on the right hand side lies on $\mE_f$.  Near $P_\infty$, we express $\psi'$ in terms of the coordinates $(1/S,T/S)$ around $P_\infty$ and the coordinates 
$(W/Y^2,1/Y)$ around $P_+$ by $(1/S,T/S) \mapsto$
\[ (W/Y^2,1/Y)=(\frac{z_0(\tfrac{z_0-1}{S}-\tfrac{T}{S})^2 -(z_0-1)^2 \tfrac{T}{S} \tfrac{1}{S}-(z_0-1) (\tfrac{T}{S})^2 -\tfrac{T^3}{S^2}    }{(1-z_1/S)^2}-1, \frac{\tfrac{z_0-1}{S}-\tfrac{T}{S}}{1-z_1/S}).\]
Writing the equation of $\mE_r$ as
\[ \tfrac{T^3}{S^2}=1+\tfrac{r^{-1}}{S^2}- (\tfrac{T}{S})^2,\]
we see that $T^3/S^2$ evaluates to $1$ at $P_\infty$. Therefore, $\psi'(P_\infty)=P_+$. \\

Near $(-z_1,z_0-1)$ we express $\psi'$ in terms of the coordinates $(W/Y^2,1/Y)$ near $P_-$ as \[ (S,T)\mapsto (W/Y^2,1/Y)=(\frac{(z_0-T)(z_0-T-1)^2}{(S-z_1)^2}-1, \tfrac{z_0-1-T}{S-z_1}). \]
Therefore $\psi'(-z_1,z_0-1)=P_-$. If $z_1=0$, this completes the definition of $\psi'$.  If $z_1 \neq 0$, then we show that  $\psi'(z_1,z_0-1)$ is the point  $(1,Y,W)=(1,\tfrac{1-z_0}{z_1}, 1-(\tfrac{1-z_0}{z_1})^2)$: as above, we use the coordinates $(\tilde S, \tilde T)=(S-z_1,T-(z_0-1))$ in terms of which  $\mE_r$ can be written as  $\tilde T^3+\tilde T^2 -\tilde S^2 +z_1 \tilde S +(1-z_0) \tilde T=0$. We have: 
\[ \frac{\tilde S}{\tilde T}= \frac{z_0-1-\tilde T- \tilde T^2}{z_1-\tilde S},\]
evaluates to $(z_0-1)/z_1$ at $(\tilde S, \tilde T)=(0,0)$, that is, $(S,T)=(z_1,z_0-1)$.
We can express $\psi'$ as 
\[ (\tilde S, \tilde T) \mapsto (1,Y,W)=(1, \tfrac{-\tilde S}{\tilde T}, 1-\tilde T-(\tfrac{\tilde S}{\tilde T})^2 ).\]
Therefore, $\psi'(z_1,z_0-1)=(1,\tfrac{1-z_0}{z_1},1-(\tfrac{1-z_0}{z_1})^2)$. \\

It is readily checked that $\psi' \circ \psi$ is the identity map on $\mE_f$ and $\psi \circ \psi'$ is the identity map on $\mE_r$, and hence $\mE_r$ and $\mE_f$ are isomorphic over $\bF_q$. We conclude that $\zeta_L=\#\mE_f(\bF_q)/2 =\#\mE_r(\bF_q)/2$.
\eep

\section{Line-Plane Incidence Numbers} \label{six}
The hyperplanes of $PG(3,q)=\bP(D_3V)$ can be identified with the dual projective space $\bP( (D_3V)^*)$. Since $D_3V = \text{Sym}^3(V^*)^*$, we identify the planes of $PG(3,q)$ with the binary cubic forms of the projective space $\bP(\text{Sym}^3(V^*))$.
The  classification of planes of $PG(3,q)$ into $G$-orbits can be found in  \cite[Corollary $4$, Lemma 21.1.3]{Hirschfeld3}:
\begin{prop}\label{plane_orbits}
There are five $G$-orbits of  planes of $PG(3,q)$ given by:
\begin{enumerate}
    \item $\mN_1=G\cdot X^3$ of size $(q+1)$, consisting of the osculating planes of $C$.
    \item $\mN_2=G\cdot X^2Y$ of size $q(q+1)$, consisting of planes which meet $C$ in exactly two points of multiplicity $1$ and $2$.
    \item $\mN_3=G\cdot XY(X-Y)$ of size $q(q^2-1)/6$, consisting of planes which intersect $C$ in  exactly three points of multiplicity $1$.
    
    \item $\mN_4=G\cdot X(Y^2-\ep X^2)$ of size $q(q^2-1)/2$, consisting of  planes which intersect $C$ in a point of multiplicity $1$, and a pair of Galois conjugate points of $C(\bF_{q^2})$ with multiplicity $1$ each.
    
    \item $\mN_5=G\cdot f(X,Y)$ of size $q(q^2-1)/3$, where $f(X,Y)$ is irreducible over $\bF_q$, consisting of  planes which meet $C(\bF_{q^3})$ in three Galois conjugate points with multiplicity $1$ each.\\
\end{enumerate}
\end{prop}

\begin{lem}\cite[Lemma 15.2.3]{Hirschfeld3} \label{line_plane}
    Let $L$ be a line of $PG(3,q)$ with Pl\"ucker coordinates $(p_{01},p_{02},p_{03},p_{12},p_{13},p_{23})$. The pencil of planes through $L$ consists of cubic forms $a_0X^3+a_1X^2Y+a_2XY^2+a_3Y^3$, such that $(a_0,a_1,a_2, a_3)$ is in the kernel of the rank $2$ alternating matrix 
    \[ \bbm
0 & p_{01} & p_{02} & p_{03} \\
-p_{01} & 0 & p_{12} & p_{13} \\
-p_{02} & -p_{12} & 0 & p_{23} \\
-p_{03} & -p_{13} & -p_{23} & 0 
\bem  = \bbm
0 & z_4 & z_3 & z_2 \\
-z_4 & 0 & z_5 & z_1 \\
-z_3 & -z_5 & 0 & z_0 \\
-z_2 & -z_1 & -z_0 & 0 
\bem.  \]

\end{lem}
We note that  the rank of the alternating matrix above is a positive even integer, and its determinant equals $(p_{01}p_{23} -p_{02}p_{13}+p_{03}p_{12})^2=0$. Therefore,  its rank is $2$.

\begin{cor}\label{pencil-planes}
    Let $L$ be a line of $PG(3,q)$ not meeting $\mA$. Let  $f_L=X^2Y^2+ z_0X^4+z_1X^3Y+z_3XY^3+z_4Y^4$ be the binary quartic form associated  with $L$. The pencil of planes through $L$ is given by $\langle z_1X^3-X^2Y+z_4Y^3, z_0X^3-XY^2+z_3 Y^3\rangle$.
\end{cor}
\begin{proof}
    Let $a_0X^3+a_1X^2Y+a_2XY^2+a_3Y^3$ be a cubic form in the pencil of planes through $L$. Then by Lemma \ref{line_plane}, $a_0,a_1,a_2,a_3,a_4$ satisfy the the following two equations
    \begin{align*}
        a_1z_4+a_2z_3+a_3&=0,\\
        a_0+a_1z_1+a_2z_0&=0.
    \end{align*}
    From these two equations, we can see that the pencil of planes through $L$ contains the cubic forms $z_1X^3-X^2Y+z_4Y^3$ and  $z_0X^3-XY^2+z_3 Y^3$. Thus the pencil of planes through $L$ can be given by $\langle z_1X^3-X^2Y+z_4Y^3, z_0X^3-XY^2+z_3 Y^3\rangle$.
\end{proof}


 \begin{lem}\cite[Corollary 1.15(ii)]{Hirschfeld2}\label{cubicr}
     Let $F(T)\in \bF_q[T]$ be a cubic polynomial over $\bF_q$ in one variable $T$, with discriminant $\Delta(F) \neq 0$. Then $F(T)$ has exactly one zero in $\bF_q$ if and only if $\Delta$ is a non-square.
\end{lem}

In the next Proposition, for $\mO$ being any of the $10$ orbits of non-generic lines of $PG(3,q)$,  we compute the numbers $\{|S \cap \mN_i| \colon  i = 1, \dots ,5\}$, where $S$ is the pencil of planes through a line $L$ representing $\mO$.
\begin{prop} \label{nongenericresult3}
    Let $L$ be a non-generic line of $PG(3,q)$ and let $S$ denote the set of planes through $L$. For $L\in \{\mO_7,\mO_2,\mO_4,\mO_{8.1}^+,\mO_{8.1}^-,\mO_{8.2},\mO_1,\mO_3,\mO_5^+,\mO_5^-\}$, the numbers $|S\cap \mN_i|$, for $i=1,\dots,5$ can be given by the following table:
\begin{table}[h] 
\begin{tabular}{c| *{5}{c}}
Line orbits & $|S\cap \mN_1|$ & $|S \cap \mN_2|$ & $|S\cap \mN_3|$ & $|S\cap \mN_4|$ & $|S\cap \mN_5|$ \\
\hline
$\mO_7$ &  $q+1$ &  $0$ & $0$ & $0$ & $0$ \\
$\mO_2$ & $1$ & $q$ & $0$ & $0$ & $0$ \\
$\mO_4$ & $1$ & $1$ & $\tfrac{q-1}{2}$ & $\tfrac{q-1}{2}$ & $0$ \\
$\mO_{8.1}^+$ & $1$ & $0$ & $\tfrac{q}{3}$ & $0$ & $\tfrac{2q}{3}$ \\
$\mO_{8.1}^-$ & $1$ & $0$ & $0$ & $q$ & $0$ \\
$\mO_{8.2}$ & $1$ & $1$ & $\tfrac{q-3}{6}$ & $\tfrac{q-1}{2}$ & $\tfrac{q}{3}$ \\
$\mO_1$ & $0$  & $2$ & $q-1$ & $0$ & $0$ \\
$\mO_3$ & $0$ & $0$ & $0$ & $q+1$ & $0$ \\
$\mO_5^+$ & $0$ & $3$ & $\tfrac{q-3}{2}$ & $\tfrac{q-1}{2}$ & $0$\\
$\mO_5^-$ & $0$ & $1$ & $\tfrac{q-1}{2}$ & $\tfrac{q+1}{2}$ & $0$
\end{tabular}
\end{table}
    
\end{prop}
\begin{proof} A line $L$ which meets $\mA$ must be contained in an osculating plane of $C$. Therefore, we may assume $L$ is contained in the osculating plane to $C$ at $\nu_3(0,1)$. In other words, we may take the Pl\"ucker coordinates of $L$ to satisfy $p_{01}=p_{02}=p_{03}=0$, that is, $z_2=z_3=z_4=0$. Therefore, the kernel of the alternating matrix above in Lemma \ref{line_plane} is spanned by $(1,0,0,0)$ and $(0,z_0,-z_1,z_5)$. In other words the pencil of cubic forms through $L$ is $\langle X^3, Y(z_0X^2-z_1XY+z_5Y^2) \rangle$.
\begin{enumerate}
    \item For $\mO_7=G \cdot(0,0,0,0,0,1)$, the pencil is $\langle X^3, Y^3 \rangle$ which is completely contained in $\mN_1$.
    
\item For $\mO_2=G \cdot (1,0,0,0,0,0)$, the pencil is $\langle X^3, YX^2 \rangle$ which has $1$ point in $\mN_1$ and $q$ points in $\mN_2$.

\item For  $\mO_4=G \cdot (0,1,0,0,0,0)$, the pencil is $\langle X^3, XY^2\rangle$ which contains $1$ point each in $\mN_1, \mN_2$ and $(q-1)/2$  points each in  $\mN_3,\mN_4$.

\item  For $\mO_{8.1}^+=G \cdot (1,0,0,0,0,1)$, the pencil is $\langle X^3, Y(X^2+Y^2)\rangle$. Apart from $X^3$ which is in $\mN_1$, the remaining $q$ points of the pencil have $q/3$ points in $\mN_3$ and $2q/3$ points in $\mN_5$ because the polynomial $y^3+y- \lambda$ has either $3$  or $0$ roots in $\bF_q$, and there are  $q/3$ and $2q/3$ values of $\lambda$ respectively for each of these two possibilities.

\item  For  $\mO_{8.1}^-=G \cdot (1,0,0,0,0,\ep)$, the pencil is $\langle X^3, Y(X^2+\ep Y^2)\rangle$. Apart from $X^3$ which is in $\mN_1$, the remaining $q$ points of the pencil are in $\mN_4$ because the polynomial $y^3+\ep y- \lambda$ has either exactly $1$  root in $\bF_q$ for each $\lambda \in \bF_q$.

\item For $\mO_{8.2}=G \cdot (0,1,0,0,0,1)$, the pencil is 
$\langle X^3, Y^2(Y-X) \rangle$. Apart from $X^3$ which is in $\mN_1$, and $Y^2(Y-X)$ which is in $\mN_2$, the remaining $(q-1)$ points of the pencil have $|J_4|, |J_2|$ and $|J_1|$ points in $\mN_3, \mN_4$ and $\mN_5$, because the polynomial $y^3-y^2+ \lambda=E_{\lambda^{-1}}(1,y)$ has either $(i-1)$ roots in $\bF_q$ when $\lambda^{-1} \in J_i$ for $i \in \{1,2,4\}$, according to Lemma \ref{Ji_lem}.
\end{enumerate}

Let $L$ be a non-generic line meeting the axis and $f_L$ be the binary quartic form associated with it. Since the coefficient $z_2$ of $X^2Y^2$ of $f_L$ is non-zero, we can take $f_L(X,Y)=X^2Y^2+(z_0X^4+z_1X^3Y+z_3XY^3+z_4Y^4)$. By Corollary \ref{pencil-planes}, the pencil of planes through $L$ is given by $\langle z_1X^3-X^2Y+z_4Y^3, z_0X^3-XY^2+z_3 Y^3\rangle$.
\begin{enumerate}
\item For $\mO_1=G\cdot (0,0,1,0,0,0)$, the pencil is $\langle X^2Y, XY^2\rangle$ which has $2$ points in $\mN_2$ and remaining $(q-1)$ points in $\mN_3$.
\item  For $\mO_3=G \cdot (\ep,0,1,0,1/\ep,-1)$,  the pencil is $\langle Y^3- \ep X^2Y, \ep X^3-XY^2\rangle$ all whose points are in $\mN_4$.

\item  For $\mO_5^+=G \cdot (0,0,1,-1,0,0)$, the pencil is  $\langle X^2Y, XY^2+ Y^3\rangle$. Apart from $X^2Y$ and $Y(Y-X)^2$ which is $\mN_2$, the remaining $q-1$ points of the pencil are $\{Y((Y-X)^2- \lambda X^2) \colon \lambda \in \bF_q^\times\}$ of which one of them (for $\lambda=1$) is in $\mN_2$, and there are $(q-3)/2$ points and $(q-1)/2$ points in $\mN_3$  and $\mN_4$, respectively.

\item For $\mO_5^-=G \cdot (0,0,1,0,-1/\ep,0)$, the pencil is $\langle Y(\ep X^2+Y^2), XY^2\rangle$. Apart from $XY^2$  which is $\mN_2$, the remaining $q$ points of the pencil are $\{Y((Y-\lambda X)^2- (\lambda^2-\ep) X^2) \colon \lambda \in \bF_q\}$ of which there are $(q-1)/2$ points and $(q+1)/2$ points in  $\mN_3$  and $\mN_4$, respectively.
\end{enumerate}
\end{proof}

\subsection{Proof of Theorem \ref{result3}}
Let $L$ be a generic line and $f_L$ be the binary quartic form associated with $L$, with $z_2(f_L)=1$ and $\jmath(f_L)=r$. Let  $\mu_L$ denotes the number of distinct linear factors of $f_L$ over $\bF_q$:
\[ \mu_L = \begin{cases} i &\text{ if $f_L \in \mF_i$ for $i \in \{1,2,4\}$},\\
0 &\text{if $f_L \in \mF_2' \cup \mF_4'$}. \end{cases} \]
We note that we may rewrite the assertion of Theorem \ref{result3} as
\begin{align*}
|S \cap \mN_1|&=0,\\
 |S \cap \mN_2|&=\mu_L,\\
|S \cap \mN_4|&=q+1 -  \tfrac{\#\mE_r(\bF_q)+ \mu_L}{2},\\
 |S\cap \mN_3|&=\tfrac{\#\mE_r(\bF_q)-3 \mu_L}{6}, \\
|S\cap \mN_5|&=\tfrac{\#\mE_r(\bF_q)}{3}.
\end{align*} 
We will prove this equivalent assertion. Let
\[ \nu_L=\#\{ (s,t) \in PG(1,q) \colon  f_L(s,t) \text{ is a non-square} \}.\]

From Corollary \ref{pencil-planes}, it follows that the pencil of planes through $L$ is given by $\langle z_1X^3-X^2Y+z_4Y^3, z_0X^3-XY^2+z_3 Y^3\rangle$. Let $F_{s,t}(T)\in \bF_q[T]$ denote the cubic polynomial $(sz_1+tz_0)T^3-sT^2-t T+(sz_4+t z_3)$. By \cite[Lemma 1.18]{Hirschfeld2}, the discriminant of $F_{s,t}(T)$ is given by 
\[
\Delta(F_{s,t})=z_0s^4+z_1s^3t+s^2t^2+z_3st^3+z_4t^4=f_L(s,t). \]
For a  line $L$ not meeting $C$, we claim that for each $(x,y) \in PG(1,q)$, the linear form $(Xy-Yx)$ divides a unique element of the pencil of cubic forms $\{F_{s,t} \colon (s,t) \in PG(1,q)\}$: to see this we note that if $(Xy-Yx)$ divides $2$ distinct elements of the pencil, then $\nu_3(x,y) \in L$, which is not the case. On the other hand, the join of $L$ and $\nu_3(x,y)$  is a plane through $L$, because $\nu_3(x,y) \notin L$, and hence $(Xy-Yx)$ must divide some element of the pencil $F_{s,t}$. We consider the set  
\[ X_S=\{(F_{s,t},Xy-Yx) \colon (s,t), (x,y) \in PG(1,q), (Xy-Yx) |F_{s,t} \}. \]
Since a generic line does not meet $C$, it follows from the above observation that for each $(x,y) \in PG(1,q)$, there is a unique 
$(s,t) \in PG(1,q)$ such that $(Xy-Yx)$ divides $F_{s,t}$, and hence $|X_S|=(q+1)$.
On the other hand $F_{s,t}$ has $1,2$ or $3$ distinct linear factors over $\bF_q$ according as $F_{s,t} \in \mN_4 \cup \mN_1, \mN_2$ or $\mN_3$. Thus $|X_S|=2|S\cap \mN_2|+3|S\cap \mN_3|+(|S\cap \mN_4|+|S\cap \mN_1|)$. Since $L$ is a generic line, it is not contained in any osculating plane of $C$, and hence $|S \cap \mN_1|=0$. Thus, we obtain the system of equations:
\begin{align*}
2|S\cap \mN_2|+3|S\cap \mN_3|+|S\cap \mN_4|&=q+1, \\
    |S\cap \mN_2|+|S\cap \mN_3|+|S\cap \mN_4|+|S\cap \mN_5|=|S|&=q+1.
\end{align*}
We note that  $\Delta(F_{s,t})=0$ if and only if $F_{s,t}$ has a repeated factor, that is, $ F_{s,t} \in \mN_1 \cup\mN_2 $. Since $F_{s,t} \notin \mN_1$, we get $|S \cap \mN_2|=\mu_L$. By Lemma \ref{cubicr}, $F_{s,t}(T)$ has exactly one linear factor in $\bF_q$ if $\Delta(F_{s,t})$ is a non-square. Therefore, $F_{s,t} \in \mN_4$ if and only if  $f_L(s,t)$ is a non-square. Therefore, $|S \cap \mN_4|=\nu_L$.
We can now solve the above linear system of equations to get 
\begin{align} \label{eq:S_cap_N_i}
|S \cap \mN_1|&=0,\\
\nonumber |S \cap \mN_2|&=\mu_L,\\
\nonumber|S \cap \mN_4|&=\nu_L,\\
\nonumber|S\cap \mN_3|&=\tfrac{(q+1)-2\mu_L-\nu_L}{3}, \\
\nonumber|S\cap \mN_5|&=\tfrac{2(q+1)-\mu_L-2\nu_L}{3}.
\end{align} 
By Theorem \ref{thm_elliptic} we have 
\[ \zeta_{f_L}=(q+1- \mu_L - \nu_L)= \frac{\#\mE_r(\bF_q)- \mu_L}{2},\]
and hence 
\[ \nu_L = q+1 -  \frac{\#\mE_r(\bF_q)+ \mu_L}{2}.\]
Using this in \eqref{eq:S_cap_N_i}, we get 
\begin{align*}
|S \cap \mN_1|&=0,\\
 |S \cap \mN_2|&=\mu_L,\\
|S \cap \mN_4|&=q+1 -  \tfrac{\#\mE_r(\bF_q)+ \mu_L}{2}\\
 |S\cap \mN_3|&=\tfrac{\#\mE_r(\bF_q)-3 \mu_L}{6}, \\
|S\cap \mN_5|&=\tfrac{\#\mE_r(\bF_q)}{3}.
\end{align*}

\section{Point-Line Incidence Numbers} \label{seven}
We begin by recalling the $G$-orbits of points of $PG(3,q)$:
\begin{prop}\cite[Corollary 5(ii), Lemma 21.1.3]{Hirschfeld3} \label{point_orbits}
There are five $G$-orbits for the points of $PG(3,q)$ given by 
\begin{enumerate}
    \item $\mC=G \cdot (0,0,0,1)$ consisting of  the $(q+1)$ points of the twisted cubic $C$.
    \item $\mAx=G \cdot (0,0,1,0)$ consisting of  the $(q+1)$ points of the axis $\mA$.
    \item $\mT=G \cdot (0,0,1,1)$consisting of  the  $(q^2-1)$ points lying on tangent lines of $C$ but not on $C \cup \mA$.
    \item $\mRC=G \cdot (1,0,0,1)$ consisting of  the  $q(q^2-1)/2$ points lying on real chords of $C$ but not on $C$.
    \item $\mIC=G \cdot (1,0,\epsilon,0)$ consisting of  $q(q^2-1)/2$ points of imaginary chords of $C$.    
\end{enumerate}
\end{prop}


\begin{lem}\label{tangential}
Given a point $w=\sum_{i=0}^3 y_i B_i\in D_3V$, we consider the quadratic form $P_w(X,Y) = (y_2^2-y_1y_3)X^2 + (y_1y_2-y_0y_3)XY + (y_1^2-y_0y_2) Y^2 \in \text{Sym}^2(V^*)$. We have $P_w(X,Y)$ is the zero polynomial if and only if $w \in C$, and for $w \notin C$:
\begin{enumerate}
    \item $w \in \mT \cup \mAx$ if and only if $P_w$ is the square of a linear form,
    \item $w \in \mRC$ if and only if $P_w$ is a product of two independent linear forms over $\bF_q$,
    \item $w \in \mIC$ if and only if $P_w$ is irreducible over $\bF_q$.
\end{enumerate}
  
\end{lem}
\begin{proof}
    First we show that  $P_w$  has the property that   $ g \cdot P_w = \det(g)^{-4} P_{g \cdot w}$ for for $g =\bbsm a & b\\ c& d \besm \in GL_2(q)$. Here $g$ acts on the coordinates of $w$ by the  matrix $g^{[3]}$ (see \eqref{eq:g_3}) and on $P_w(X,Y)$ by $g \cdot P = \det(g)^{-2} P_w(dX-bY, aY-cX)$ as defined in \S \ref{two}.  Since $GL_2(q)$ is generated by the matrices of the form  $\bbsm 1 & 0 \\ 0 & d \besm, \bbsm 0 & 1\\ 1 & 0 \besm$, and $\bbsm 1 & 0 \\ c & 1 \besm$, it is enough  to prove the property only for such matrices. For $g=\bbsm 1 & 0 \\ 0 & d \besm$, we have then $g^{[3]}(y_0,y_2,y_2,y_3)= (y_0,dy_1,d^2y_2,d^3y_3)$, and hence  $P_{ g \cdot w}=d^2P_w(dX,Y) = d^4  g \cdot P_{w}$.  For  $g=\bbsm 0 & 1 \\ 1 & 0 \besm$, we have  $g^{[3]}(y_0,y_1,y_2,y_3)=(y_3,y_2,y_1,y_0)$, and hence  $P_{ g\cdot w} = P_w(Y,X) = g\cdot P_w$.  For $g=\bbsm 1 & 0 \\ c & 1 \besm$, we have $g^{[3]}(y_0,y_1,y_2,y_3)=(y_0,cy_0+y_1,c^2y_0-cy_1+y_2,c^3y_0+y_3)$, and it is readily checked that $P_{ g \cdot w} = g \cdot P_w$.
    
Since $C$ is defined by the quadrics $Y_0Y_2-Y_1^2, Y_0Y_3-Y_1Y_2$ and $Y_1Y_3-Y_2^2$, we see that $P_w=0$ if and only if $w \in C$. For $w \notin C$, we note that for the representatives 
$(0,0,1,0)$ and  $(0,0,1,1)$ of $\mAx$ and $\mT$, we have $P_w$ equals $X^2$, where as for the representatives  $(1,0,0,1)$ and $(1,0,\epsilon,0)$ of $\mRC$ and $\mIC$, we have $P_w$ equals $-XY$ and $\ep( \ep X^2-Y^2)$ respectively. Thus $P_w$ has $0, 1$ or  $2$ independent  linear factors over $\bF_q$ if and only if $w\in \mIC, \mT \cup \mAx$ or $\mRC$ respectively.
\end{proof}


In the next lemma, we  use Lemma \ref{tangential}, to calculate the numbers $|\mP\cap \mM|$, where $\mP$ denotes the set of $(q+1)$ points of a non-generic line $L$ and $\mM$ denotes one of the orbit $\{\mC,\mAx,\mT,\mRC,\mIC\}$.

\begin{lem}
  Let $L$ be a non-generic line and let $\mP$ denote the points of $L$. For $L\in \{\mO_7,\mO_2,\mO_4,\mO_{8.1}^+,\mO_{8.1}^-,\mO_{8.2},\mO_1,\mO_3,\mO_5^+,\mO_5^-\}$, the numbers $|\mP\cap \mM|$, for $\mM\in \{\mC,\mAx,\mT,\mRC,\mIC\}$ can be given by the following table:
\begin{table}[h] 
\begin{tabular}{c| *{5}{c}}
Line orbits & $|\mP \cap \mC|$ & $|\mP \cap\mAx|$ & $|\mP \cap \mT|$ & $|\mP \cap \mRC|$ & $|\mP \cap \mIC|$ \\
\hline
$\mO_7$ & $0$ & $q+1$ & $0$ & $0$ & $0$ \\
$\mO_2$ & $1$ & $1$ & $q-1$ & $0$ & $0$ \\
$\mO_4$ & $1$ & $1$ & $0$ & $\tfrac{q-1}{2}$ & $\tfrac{q-1}{2}$ \\
$\mO_{8.1}^+$ & $0$ & $1$ & $0$ & $q$ & $0$ \\
$\mO_{8.1}^-$ & $0$ & $1$ & $0$ & $0$ & $q$ \\
$\mO_{8.2}$ & $0$ & $1$ & $1$ & $\tfrac{q-1}{2}$ & $\tfrac{q-1}{2}$\\
$\mO_1$ & $2$ & $0$ & $0$ & $q-1$ & $0$ \\
$\mO_3$ & $0$ & $0$ & $0$ & $0$ & $q+1$ \\
$\mO_5^-$ & $1$ & $0$ & $0$ & $\tfrac{q-1}{2}$ & $\tfrac{q+1}{2}$\\
$\mO_5^+$ & $1$ & $0$ & $2$ & $\tfrac{q-3}{2}$ & $\tfrac{q+1}{2}$\\
\end{tabular}
\end{table}
  \end{lem}

\begin{proof}
(1) If $L\in \mO_7$ is the axis then it all its $(q+1)$ points are in $\mAx$.

(2) If  $L\in \mO_2$ is a tangent line, then it has one point each in $\mC$ and $\mAx$ and the remaining $(q-1)$ points are in $\mT$.

(3) If  $L\in \mO_4$ be a non-tangent unisecant meeting the axis. Let $L$ be the joining of the two points $(0,0,0,1)$ and $(0,1,0,0)$. Apart from these $2$ points which lie on $\mC$ and $\mAx$, respectively, the remaining $(q-1)$ points of this line are $\{ w(s)=(0,s,0,1) \colon s\in \bF_q^\times\}$ which has  $(q-1)/2$ points each in $\mRC$ and $\mIC$, because $P_{w(s)}=s(s Y^2 - X^2)$.

(4) If $L \in \mO_{8.1}^+$ is an external line lying on an osculating plane, represented by the line joining  $(0,0,1,0)$ and $(0,1,0,1)$.  Apart from $(0,0,1,0) \in \mAx$, the remaining $q$ points of this line are  $\{ w(s)=(0,1,s,1) \colon s \in \bF_q\}$. These $q$ points are in $\mRC$ because  $P_{w(s)}= (Y-(s+1)X) (Y-(s-1) X)$.

(5) If  $L\in \mO_{8.1}^-$ is  an external line lying on an osculating plane, we can take  $L$ to be the line  joining of the two points $(0,0,1,0)$ and $(0,\ep,0,1)$.  Apart from $(0,0,1,0) \in \mAx$, the remaining $q$ points of this line are  
$\{ w(s)= (0,\ep,s,1) \colon s\in \bF_q\}$.
These $q$ points are in $\mIC$ because 
$P_{w(s)}=(\ep  Y- s X)^2 - \ep X^2$ is irreducible over $\bF_q$.

(6) If  $L\in \mO_{8.2}$ is an external line lying on an osculating plane, we can take $L$ to be the line  joining of the two points $(0,0,1,1)$ and $(0,1,0,0)$. Apart from $(0,1,0,0) \in \mAx$ and $(0,0,1,1) \in \mT$,  the remaining $(q-1)$ points of this line are  
$\{ w(s)= (0,1,s,s) \colon s\in \bF_q^\times\}$. Of these $(q-1)$ points, there are $(q-1)/2$ points each in $\mRC$ and $\mIC$ because 
$P_{w(s)}=
(s^2-s)X^2 + sXY +  Y^2 = (Y-sX)^2 -s X^2$.

(7) If  $L\in \mO_1$ is a real chord of $C$, we  can take $L$ to be the line  joining of the two points $(1,0,0,0)$ and $(0,0,0,1)$. Apart from $(1,0,0,0), (0,0,0,1)  \in \mC$,  the remaining $(q-1)$ points of this line are  
$\{ w(s)= (1,0,0,s) \colon s\in \bF_q^\times\}$. These $(q-1)$ points are in $\mRC$ because $P_{w(s)}=-sXY$ is a product of two independent linear factors.

(8) If $L\in \mO_3$ is an imaginary chord not,
 we can take $L$ to be the line  joining of the two points $(1,0,\ep,0)$ and $(0,1,0,\ep)$. The $(q+1)$ points of this line are  
$\{ w(s)= (1,s,\ep,s\ep) \colon s\in \bF_q\cup \{\infty\}\}$. These $(q+1)$ points are in $\mIC$ because $P_{w(s)}=(s^2-\ep)(Y^2-\ep X^2)$ is irreducible over $\bF_q$.

(9)  Let $L\in \mO_5^-$ be a non-tangent unisecant not lying on an osculating plane. We can take $L$ to be the line  joining of the two points $(1,0,0,0)$ and $(0,1,0,-\ep)$.  We can see that only the point $(1,0,0,0)$ of $L$ is in $\mC$. Apart from this point the remaining $q$ points of this line are $\{ w(s)= (s,1,0,-\ep) \colon s\in \bF_q\}$. Here $P_{w(s)}=\ep X^2+\ep sXY+Y^2=(Y-sX)^2-(s^2-\ep) X^2$, which is irreducible for $(q+1)/2$ values of $s$ and a product of two independent linear forms for $(q-1)/2$ values of $s$.\\

(10) If  $L\in \mO_5^+$ is a non-tangent unisecant not lying on an osculating plane,  we can take $L$ to be the line  joining of the two points $(1,0,0,0)$ and $(0,0,1,-1)$. These $2$ points are in $\mC$ and $\mT$, respectively, and the remaining $(q-1)$  points of this line are are $\{ w(s)= (1,0,s,-s) \colon s\in \bF_q^\times\}$. Here
$P_{w(s)}=s((s+1)X^2-(X+Y)^2)$ which is a square of a linear form for $s=-1$, a product of two independent linear forms  for $(q-3)/2$ values of $s$, and irreducible  for $(q+1)/2$ values of $s$.
 Since $L$ contains no points of $\mAx$, the point  $(-1,0,1,-1)$ for $s=-1$ is in $\mT$. 
\end{proof}

\begin{lem}\label{intersection}
    Let $L$ be a line with Pl\"ucker coordinates $(p_{01},p_{02},p_{03},p_{12},p_{13},p_{23})$, which does not meet the axis. Then the point of intersection of $L$ with the  osculating plane $O_{(s,t)}$ is given by $(p_{03}s^3,p_{13}s^3+p_{01}t^3,p_{23}s^3+p_{02}t^3,p_{03}t^3)$.
\end{lem}
\begin{proof}
Since $L$ is not contained in any osculating plane of $C$, it meets each osculating plane of $C$ in a unique point. A point $(y_0,y_1,y_2, y_3)$ lies on the  osculating plane $O_{(s,t)}$ if and only if 
$t^3y_0-s^3y_3=0$. The condition for a point  $(y_0, y_1, y_2,y_3)$ to lie on the line  $L$
is     \[ (y_0,y_1,y_2, y_3) \bbm 0 & -p_{23} & p_{13} & -p_{12} \\
p_{23} & 0 & -p_{03} & p_{02} \\
-p_{13} & p_{03} & 0 & -p_{01} \\
p_{12} & -p_{02} & p_{01} & 0 
\bem=0,\] by \cite[Lemma 15.2.3]{Hirschfeld3}. Both these conditions are satisfied by $P=(p_{03}s^3,p_{13}s^3+p_{01}t^3,p_{23}s^3+p_{02}t^3,p_{03}t^3)$, and hence $P$ is the point of intersection of  $L$ and $O_{(s,t)}$.
\end{proof}

\subsection{Proof of Theorem \ref{result4}}
Let $L$ be a generic line and $f_L$ be the binary quartic form associated with $L$, with $z_2(f_L)=1$ and $\jmath(f_L)=r$. Let  $\eta_L$ denote the number of distinct linear factors of $f_L$ over $\bF_q$:
\[ \eta_L = \begin{cases} i &\text{ if $f_L \in \mF_i$ for $i \in \{1,2,4\}$,}\\
0 &\text{if $f_L \in \mF_2' \cup \mF_4'$}. \end{cases} \]
We note that we may rewrite the assertion of Theorem \ref{result4} as
\begin{align*}
|\mP \cap \mC|&=0,\\
|\mP \cap \mAx|&=0,\\
 |\mP \cap \mT|&=\eta_L,\\
|\mP \cap \mRC|&=\tfrac{\#\mE_r(\bF_q)-\eta_L}{2},\\
 |\mP\cap \mIC|&=q+1-\tfrac{\#\mE_r(\bF_q)+\eta_L}{2}.
\end{align*} 

We will prove this equivalent assertion. The first two equalities of the above assertion follow immediately from the fact that $L$ intersects neither $C$ nor the axis $\mA$. Let
\[ \zeta_L=\#\{ (s,t) \in PG(1,q) \colon  f_L(s,t) \text{ is a non-zero square} \}.\]

By Lemma \ref{intersection}, the line $L$ with Pl\"ucker coordinates  $p_{ij}$, $0\leq i<j\leq 3$,  meets the osculating plane $O_{(s,t)}$ in the unique point $w(s,t)=(p_{03}s^3,p_{13}s^3+p_{01}t^3,p_{23}s^3+p_{02}t^3,p_{03}t^3)$. Any two osculating planes meet in $\mA$, and $L$ does not meet $\mA$. Therefore, the set of $(q+1)$ points $\mP$ of $L$  is precisely $\{w(s,t) \colon (s,t) \in PG(1,q) \}$. Now $P_{w(s,t)}(X,Y)=(y_2^2-y_1y_3)X^2 + (y_1y_2-y_0y_3)XY + (y_1^2-y_0y_2) Y^2$, where $(y_0,y_1,y_2,y_3)=(p_{03}s^3,p_{13}s^3+p_{01}t^3,p_{23}s^3+p_{02}t^3,p_{03}t^3)$.

By lemma \ref{tangential}, the point $w(s,t)$ will be in $ \mT,\mRC$ or $ \mIC$ according as the discriminant $\Delta(P_{w(s,t)})$ of the quadratic form $P_{w(s,t)}$ is $0$, a nonzero square or a non-square, respectively. We calculate this discriminant 
\begin{align*}
 \Delta(P_{w(s,t)}) &=(y_1y_2-y_0y_3)^2-(y_2^2-y_1y_3)(y_1^2-y_0y_2)\\
 &=y_0^2y_3^2+y_0y_2^3+y_1^3y_3\\
 &=(p_{03}s^3)^2(p_{03}t^3)^2+p_{03}s^3(p_{23}s^3+p_{02}t^3)^3 +(p_{13}s^3+p_{01}t^3)^3p_{03}t^3\\
 &=p_{03}^2s^6t^6+p_{03}p_{23}^3s^{12}+p_{03}p_{02}^3s^3t^9+p_{03}p_{13}^3s^9t^3+p_{03}p_{01}^3t^{12}\\
 &=p_{03}(p_{23}^3s^{12}+p_{13}^3s^9t^3+p_{03}^3s^6t^6+p_{02}^3s^3t^9+p_{01}^3t^{12}) \\
 &=p_{03}(p_{23}s^4+p_{13}s^3t+p_{03}s^2t^2+p_{02}st^3+p_{01}t^4)^3\\
 &=(z_0s^4+z_1s^3t+s^2t^2+z_3st^3+z_4t^4)^3\\
 &= f_L(s,t)^3.
\end{align*}
Therefore, 
\begin{enumerate}
    \item[(i)] $|\mP\cap \mT|=\eta_L$,
    \item[(ii)]  $|\mP\cap \mRC|= \zeta_L$,
    \item[(iii)] $|\mP\cap \mIC|=(q+1)-\eta_L-\zeta_L$.
\end{enumerate}
By Theorem \ref{thm_elliptic}, we also have
\[ \zeta_L= \zeta_{f_L} =\tfrac{\#\mE_r(\bF_q)-\eta_L}{2}.\] 
Hence, 
\begin{align*}
|\mP\cap \mRC|&=\tfrac{\#\mE_r(\bF_q)-\eta_L}{2},\\
    |\mP\cap \mIC|&=q+1-\tfrac{\#\mE_r(\bF_q)+\eta_L}{2}.
\end{align*}

\bibliographystyle{amsplain}

\end{document}